\documentclass[letterpaper]{article}

\usepackage{amsfonts, amsmath, amsthm}
\usepackage{graphicx}
\usepackage{epstopdf}
\usepackage{algorithm}
\usepackage{algorithmic}
\usepackage{subcaption}
\usepackage{mathtools}
\usepackage{commath}
\usepackage{comment}
\usepackage{hyperref}
\usepackage[margin=2cm]{geometry}

\graphicspath{{./figs/}}
\ifpdf
\DeclareGraphicsExtensions{.eps,.pdf,.png,.jpg}
\else
\DeclareGraphicsExtensions{.eps}
\fi

\newtheorem{theorem}{Theorem}[section]
\newtheorem{lemma}[theorem]{Lemma}

\newtheorem{remark}[theorem]{Remark}
\newtheorem{corollary}[theorem]{Corollary}
\newtheorem{assumption}[theorem]{Assumption}
\theoremstyle{definition}
\newtheorem{definition}[theorem]{Definition}

\author{Nir Sharon\footnotemark[3], 
	\and Rafael Sherbo Cohen\thanks{School of Mathematical Sciences, Faculty of Exact Sciences, Tel-Aviv University, Israel (\email{rafaels2@mail.tau.ac.il}, \email{nsharon@tauex.tau.ac.il})}, 
	\and Holger Wendland\thanks{Faculty of Mathematics, Physics and Computer Sciences, University of Bayreuth, Germany (\email{holger.wendland@uni-bayreuth.de})}}

\usepackage{amsopn}

\newcommand{\Log}[2]{\operatorname{Log}_{#1}\left(#2 \right)}
\newcommand{\Exp}[2]{\operatorname{Exp}_{#1}\left(#2 \right)}

\newcommand{\rest}[2]{{ \left. #1 \right|_{#2} }}

\usepackage{enumitem}
\setlist[enumerate]{leftmargin=.5in}
\setlist[itemize]{leftmargin=.5in}

\def\R{\mathbb{R}}

\def\N{\mathbb{N}}

\def\cM{\mathcal{M}}

\def\SPD{\operatorname{SPD}}
\def\SO{\operatorname{SO}}

\def\00{\mathbf{0}}
\def\11{\mathbf{1}}
\def\22{\mathbf{2}}
\def\33{\mathbf{3}}
\def\44{\mathbf{4}}
\def\55{\mathbf{5}}
\def\66{\mathbf{6}}
\def\77{\mathbf{7}}
\def\88{\mathbf{8}}
\def\99{\mathbf{9}}

\def\bbeta {{\boldsymbol{\beta}}}

\DeclareMathOperator*{\argmin}{arg\,min}
\DeclareMathOperator{\supp}{supp}

\usepackage{color}
\newcommand{\rev}[1]{{\color{black}{#1}}}
\newcommand{\hw}[1]{{\color{black}{#1}}}

\title{On multiscale quasi-interpolation of scattered \\ scalar- and manifold-valued functions}
\author{Nir Sharon, Rafael Sherbu Cohen, Holger Wendland}
\date{}

\begin{document}


\maketitle

\begin{abstract}
We address the problem of approximating an unknown function from its
discrete samples given at arbitrarily scattered sites. This problem is
essential in numerical sciences, where modern applications also
highlight the need for a solution to the case of functions with
manifold values. In this paper, we introduce and analyze a combination
of kernel-based quasi-interpolation and multiscale approximations for
both scalar- and manifold-valued functions. While quasi-interpolation
provides a powerful tool for approximation problems if the data is
defined on infinite grids, the situation is more complicated when it
comes to scattered data. \hw{Here, higher-order quasi-interpolation
  schemes either require derivative information or become numerically
  unstable.} Hence, this paper \rev{principally studies} the
improvement achieved by combining quasi-interpolation with a
multiscale technique. The main contributions of this paper are as
follows. First, we introduce the multiscale quasi-interpolation
technique for scalar-valued functions. Second, we show how this
technique can be carried over \rev{using moving least-squares
  operators} to the manifold-valued setting. Third, we give a
mathematical proof that converging quasi-interpolation will also lead
to converging multiscale quasi-interpolation. Fourth, we provide ample
numerical evidence that multiscale quasi-interpolation has superior
convergence to quasi-interpolation. In addition, we will provide
examples showing that the multiscale quasi-interpolation approach
offers a powerful tool for many data analysis tasks, such as denoising
and anomaly detection. It is especially attractive for cases of
massive data points and high dimensionality.   
\end{abstract}

%

\section{Introduction}

We consider the problem of approximating an unknown function from its discrete samples. This classical problem is at the backbone of many scientific and engineering questions. For example, the samples could represent a physical quantity, a part of a biological system, or a simulation process. As it turns out, the scattered data sites, especially when dealing with \rev{a large amount of data points of high dimensionality}, make this reconstruction problem extremely challenging. \rev{In} this paper, we are interested in approximation methods that allow us to approximate scalar-valued functions and functions that take their values in manifolds. We start by describing the former.

We define our scalar-valued approximation problem as follows. Let $X=\{ x_{j} :
j\in J\}\subseteq \Omega \subseteq \R^{d}$ be data sites and $\{ v_{j}
\}_{j \in J} \subseteq \R$ be the corresponding data values. Here, $J$
is a possibly infinite index set and $\Omega \subseteq \R^{d}$ is the
region of interest. We assume that the data values are samples,
possibly in the presence of noise, of an unknown smooth enough
function $f$, that is $v_{j}=f(x_{j})+\varepsilon_{j}$, where
$\varepsilon_{j}$ is the noise term. Under these settings, a standard
approach is to use a kernel $K \colon \Omega \times \Omega
\rightarrow \R$ and its induced approximant, 
\begin{equation} \label{eqn:kernel_app} 
s(x)=\sum_{j \in J} \alpha_{j} K(x, x_j), \qquad x \in \Omega.
\end{equation}
Often, the kernel is translation invariant or even radial, meaning it
has either the form $K(x,y)=\Phi(x-y)= \phi(\norm{x-y}_2)$ with
$\Phi:\R^d\to\R$ and $\phi \colon
[0,\infty)\to\R$, \hw{respectively, where $\norm{\cdot}_2$} is the Euclidean norm in
  $\R^{d}$.

  Determining the coefficients $\alpha_j$ is then the main
  barrier in constructing the approximation. A natural way to
  determine the coefficients of $s$ 
of~\eqref{eqn:kernel_app} is by interpolation, which requires $s$ to
satisfy $s(x_j)=v_{j}$, $j\in J$. \hw{If $J$ is finite, this leads to
  a linear system $A\alpha=v$ with matrix $A=(K(x_i,x_j))\in\R^{N\times N}$ with
  $N=|J|$. Such systems are often full and hence become numerically too
  expensive to solve for larger $N$. A possible remedy to this is to
  assume that $K$ is translation-invariant,
  i.e. of the form $K(x,y)=\Phi(x-y)$ and that $\Phi$ has compact
  support. Then, for each $1\le i\le N$, the number of data sites that
  fit inside the support of $\Phi(x_i-\cdot)$ determines the number of
  non-zero entries in row $i$ of the matrix $A$. Consequently, if the
  support radius of $\Phi$ is chosen such that the number of non-zero
  entries per row is small, then we essentially have a sparse system,
  which is also known to be well-conditioned. For quasi-uniform data
  sets this is possible by choosing the support radius proportional to
  the so-called mesh-norm or fill distance of the data set. Unfortunately, it is
  also known that adding more and more points to the system while
  keeping the number of non-zero entries fixed by scaling does not
  lead to a converging process. For convergence it is necessary to
  keep the support radius fixed, meaning that eventually the matrix
  $A$ fills up again and the advantage of sparsity and
  well-conditioning is lost.} 

A remedy to this problem has been suggested
in~\cite{schaback1995creating, floater1996multistep, chen2002multilevel}, and the
convergence has first been proven in~\cite{gia2010multiscale,
  le2017zooming, le2014data, wendland2010multiscale}. \hw{Recent
  publication comprise for example \cite{Hubbert-Levesley-17-1,
    Usta-Levesley-18-1, Wendland-17-1}}. The idea is
based on the assumption of a sequence of
increasingly denser, usually nested, data sets $X_1, X_2,\ldots,$ and
a simple residual correction scheme. The process starts when the
target function is approximated on the coarsest grid $X_1$. Next, the
error of the first step is \hw{computed} on $X_2$ \hw{and then
  approximated} using a smaller
support radius of \hw{$\Phi$} and so on. This approach has not only the
advantages of sparse, well-conditioned interpolation matrices in each
step and fast, stable evaluations; it also allows us to capture
different scales, if present, in the target function. We refer to this
approach as \textit{multiscale}.

\hw{Another computational alluring}
alternative to interpolation is {\em quasi-interpolation}, where the
coefficients of $s$ of~\eqref{eqn:kernel_app} are set as the values of
samples, $\alpha_{j} = v_{j}$. The quasi-interpolation concept
\hw{often also requires polynomial reproduction of a certain
  degree. It }is used
successfully in various areas and is well studied, particularly in
spline spaces, see, e.g.,~\cite[Chapter
  12]{de1990quasiinterpolants} \hw{and
  \cite{Buhmann-Jaeger-22-1}}. \hw{Quasi-interpolation on scattered 
  data is also possible, see for example ~\cite{Buhmann-etal-95-1,feng2009shape,
    wang2010kind}. One popular
  multivariate technique is given by {\em
    moving least-squares}, see for example
  \cite{Levin-98-1,wendland2001local}. However, particular for higher
  approximation orders the numerical stability depends heavily on the
  geometry of the data sites and is  often practically not achievable.}

Here, we suggest using quasi-interpolation in a residual correction
scheme, that is, in a multiscale fashion. Our quasi-interpolation
consists of a compactly supported radial basis function
(RBF)~\cite{wendland1998error}. So the resulted scheme provides an
appealing computational technique to process massively large datasets
in high dimensions. While similar approaches were \rev{mainly} studied numerically
and in slightly different settings~\cite{fasshauer2007iterated,
  ling2004univariate, cao2015spherical}, we suggest both numerical treatment and
analytic consideration. \hw{Moreover, one of the main contributions of
  this paper is to carry the concepts over to manifold-valued functions.}  

Manifold-valued functions are of the form $F \colon \Omega \to \cM$,
where $\cM$ is a Riemannian manifold. The simplicity and efficiency of
the quasi-interpolation when equipped with a compactly supported
$\Phi$, makes this method particularly appealing for approximating
manifold-valued functions, see, e.g., \cite{grohs2017scattered}. The
manifold expresses both a global nonlinear structure together with
local, constrained, high-dimensional elements. However, classical
computational methods for approximation cannot cope with
manifold-valued functions due to manifolds' non-linearity, and even
fundamental tasks like integration, interpolation, and regression
become challenging, see, e.g.,~\cite{blanes2017concise,
  iserles2000lie, zeilmann2020geometric}.   


The problem of approximating functions with manifold values has risen
in various research areas, ranging from signal
processing~\cite{bouchard2018riemannian} and modern
statistics~\cite{petersen2019frechet} to essential applications such
as brain networks and autism
classification~\cite{fletcher2018riemannian}, structural dynamics and
its application for aerodynamic problem
solving~\cite{amsallem2009method, lieu2006reduced}, to name a few. For
example, in reduced-order models for
simulations~\cite{benner2015survey, guo2018reduced} they drastically
decrease the calculation time of simulating processes using
interpolation of nonlinear structures at scattered locations and
within high levels of accuracy~\cite{rama2018towards}.  

In this paper, we show how to use our multiscale quasi-interpolation
approach for manifold-values approximation based on the Riemannian
center of mass. First, we provide a rigorous theoretical discussion,
followed by a comprehensive numerical study that includes both the
fine details of implementation together with an illustration of the
multiscale approximation over scattered manifold data and its
application. The alluring computational nature of the
quasi-interpolation multiscale becomes even more essential for
approximating manifold-values functions as the dimensionality and
complexity of the manifold must be considered. In the numerical part,
we provide comparison tests between the direct quasi-interpolation
approximation and the multiscale approach and show its attractivity
for an application of manifold-data processing. 

The paper is organized as follows. Section~\ref{sec:notation} is
dedicated to introducing essential notation, definitions, and
background. In Section~\ref{sec: error_est}, we prove new error bounds
for quasi-interpolation of scalar-valued functions. Then,
Section~\ref{sec: multiscale} discusses the combination of
quasi-interpolation and 
multiscale techniques in the scalar-valued setting. Here, we prove a
very preliminary result, which shows that multiscale
quasi-interpolation converges at roughly the same rate as
quasi-interpolation itself, if the latter converges.  In Section~\ref{sec:
  transition_to_manifold}, we show how the construction of the
multiscale quasi-interpolation method is formed for manifold data and
provide the algorithm and a 
theoretical discussion. We summarize the manuscript with the numerical
results that are given in Section~\ref{sec: numerical}. We first
compare the performance of the multiscale quasi-interpolation to
standard quasi-interpolation. Here, the numerical examples show two
extraordinary 
features of multiscale quasi-interpolation. On the one hand they show
that non-converging quasi-interpolation can become convergent if combined
with the multiscale technique. On the other hand, they show that
the speed of convergence of quasi-interpolation can be improved by
combining quasi-interpolation  with the multiscale approach. There is some
theoretical evidence for the first case in \cite{Franz-Wendland-22-1},
though the situation described there does not apply to our scattered
data problems and hence needs further theoretical backing. The same is
true for the latter case. While it is easy to see that multiscale
quasi-interpolation must converge if quasi-interpolation converges,
see Section \ref{sec: multiscale}, there is still no rigorous proof
for this yet.
Finally, we study other aspects that highlight the
advantages of the multiscale approach. Finally, we present the
multiscale approximation for manifold data, including its application
for denoising a field of rotations that is contaminated with noise.  

\section{Preliminaries} \label{sec:notation}

We call an open, nonempty set
$\Omega\subseteq\R^d$ a {\em domain}. Noting, however, that the term
domain often also includes connectivity. We will look at classes of
differentiable functions.  

\begin{definition} Let $\Omega\subseteq\R^d$ be a bounded domain. Let
  $k\in\N_0$. The spaces of $k$-times continuously differentiable
  functions are defined as
  \begin{eqnarray*}
    C^k(\Omega)&:=& \{u\colon \Omega\to\R : D^\alpha \, \, u\in C(\Omega),
    |\alpha|\le k\},\\
    C^k(\overline{\Omega})&:=& \{u\in C^k(\Omega) : D^\alpha u \mbox{
      has a continuous extension to $\partial\Omega$ }, |\alpha|\le
    k\}.
  \end{eqnarray*}
 On the latter space we define the norm 
   \begin{equation} \label{eqn:smoothness_desc_real_values}
       \|u\|_{C^{\hw{k}}(\overline{\Omega})}:=\max_{|\alpha|\le \hw{k}}
  \sup_{x\in\overline{\Omega}} |D^\alpha u(x)|, \qquad u\in
  C^{\hw{k}}(\overline{\Omega}).
   \end{equation}
\end{definition}

We will first collect the necessary material on quasi-interpolation. In the following, $\pi_m(\R^d)$ denotes the space of all $d$-variate polynomials of (total) degree less than or equal to $m\in\N_0$. For quasi-interpolation, we use the moving least squares method.
\begin{definition}
  Let $X=\{x_1,\ldots,x_N\}\subseteq\Omega\subseteq\R^d$ be
  given. Then, the {\em Moving Least-Squares} (MLS) approximation $Q_X(f)$ of degree
  $m\in\N_0$ to a function $f\in C(\Omega)$ is defined as follows. We
  choose {\em weight functions} $w_i:\Omega\to\R$, $1\le i\le
  N$. For each $x\in\Omega$ we set $Q_X(f)(x):=p^*(x)$, where
  $p^*\in\pi_m(\R^d)$ is the solution of
\begin{equation}
    \label{eq: moving least squares}
  \min\left\{\sum_{i=1}^N |f(x_i)-p(x_i)|^2 w_i(x) :
  p\in\pi_m(\R^d)\right\}.
  \end{equation}
\end{definition}

Next, we will assume that the weight functions $w_i:\Omega\to\R$ are
compactly supported and introduce the index sets
\[
I(x):=\{j\in\{1,\ldots,N\} : w_j(x)\ne 0\}.
\]

We present the quasi-interpolation operator and its polynomial reproduction property.

\begin{definition} \label{def:poly_rep}
  Let $Q_X$ be a sample-based approximation functional of the form,
    \begin{equation}\label{eq: general_scheme}
    Q_X(f) = \sum_{i=1}^N f(x_i) a_i,
  \end{equation}
  where $a_i$ are weight functions. We say that $Q_X$
  \textit{reproduces polynomials up to degree $m$} if for all
  $p\in\pi_m(\R^d)$, we have $Q_X(p)=p$, i.e.
    \begin{equation}\label{03-mlspr}
    p(x) =  \sum_{i=1}^N p(x_i)a_i(x), \qquad x\in\Omega.
    \end{equation}
\end{definition}

Existence and certain properties of the MLS approximant are summarized
in the following theorem. Its proof can be found, for example, in
\cite{wendland_2004}. In its formulation we use the concept of 
$\pi_m(\R^d)$-unisolvent sets. A set $X$ is called
$\pi_m(\R^d)$-unisolvent, if the only function $p\in\pi_m(\R^d)$ that
vanishes on all points from $X$ is the zero function.

\begin{theorem}\label{thm:mlsprop}
 Let $X=\{x_1,\ldots,x_N\}\subseteq\Omega\subseteq\R^d$. Let
  $w_i\in C(\Omega)$, $1\le i\le N$, be non-negative.  Assume that for each $x\in\Omega$
  the set $X(x)=\{x_i : i\in I(x)\}$ is
  $\pi_m(\R^d)$-unisolvent. Then, the MLS approximation $Q_X(f)$ of \eqref{eq: general_scheme} is
  well-defined for every $f\in C(\Omega)$, and \hw{$Q_X$ reproduces
    polynomials up to} degree $m$, see
  Definition~\ref{def:poly_rep}. Moreover, there are unique 
  functions $a_1,\ldots,a_N$ for \eqref{eq: general_scheme}, that can be written as
  \begin{equation}\label{03-mlsbasis}
  a_i(x) = w_i(x) \sum_{\ell=1}^Q \lambda_\ell(x) p_\ell(x_i),
  \qquad 1\le i\le N,
  \end{equation}
  with a basis $\{p_1,\ldots,p_Q\}$ of $\pi_m(\R^d)$ and certain
    values $\lambda_1(x),\ldots, \lambda_Q(x)$, which are
    determined by the linear system
    \begin{equation}\label{03-mlssystem}
    \sum_{j=1}^Q \left(\sum_{i=1}^N w_i(x)
    p_j(x_i)p_\ell(x_i)\right)\lambda_j(x) = p_\ell(x), \qquad
    1\le \ell\le Q.
    \end{equation}
\end{theorem}

We will assume that the weight functions of~\eqref{eq: moving least squares} are
of the form
\begin{equation}\label{03-w}
w_i(x) = \Phi\left(\frac{x-x_i}{\delta_i}\right) ,
\end{equation}
with a given function $\Phi:\R^d\to\R$. Typically, this function is
either a {\em radial} function, i.e. a function of the form
\[
\Phi(x) = \phi(\|x\|_2), \qquad \phi:[0,\infty)\to\R,
\]
or of tensor product type, i.e.
\[
\Phi(x) = \prod_{j=1}^d \phi_j(\chi_j), \qquad \phi_j:\R\to\R, \quad
x=(\chi_1,\ldots,\chi_d). 
\]
For us this particular form is not important. We will, however, make
the following assumption.

\begin{assumption}\label{03-phias}
For $X=\{x_1,\ldots,x_N\}$ the weight functions $w_i:\R^d\to\R$
are given by (\ref{03-w}) with $\delta_i>0$ and a function $\Phi\in C^r(\R^d)$ having
compact support $\supp\Phi = \overline{\mathcal{B}(\00,1)}$ and being positive
on $\mathcal{B}(\00,1)=\{x\in\R^d : \|x\|_2<1\}$.
\end{assumption}

If the data set $X$ is chosen\textit{ quasi-uniform}, i.e. if the fill distance
$h_{X,\Omega}$ and the separation radius $q_X$, given by
\begin{equation} \label{eqn:fill_dist_separation_rad}
    h_{X,\Omega}=\sup_{x\in\Omega}\min_{1\le j\le N} \|x-x_j\|_2  ,
\qquad
q_X = \frac{1}{2}\min_{j\ne i} \|x_j-x_i\|_2 
\end{equation}
satisfy $q_X\le h_{X,\Omega}\le c_q q_X$ with a small constant $c_q\ge
1$  \hw{then it is usual to use the same support radius
  $\delta_i=\delta>0$ for all weight functions $w_i$. Moreover, this
  support radius can then be chosen proportional to 
$h_{X,\Omega}$.}

\begin{lemma} Let $X\subseteq\Omega$ be a quasi-uniform data set. Let
  the weights satisfy Assumption \ref{03-phias}, where the support
  radii are chosen as $\delta_i=c_\delta h_{X,\Omega}$ with a
  sufficiently large constant $c_\delta>0$. Then, the sets $X(x)$
  are $\pi_m(\R^d)$-unisolvent and the functions $a_i$ from Theorem
  \ref{thm:mlsprop} belong to $C^r(\overline{\Omega})$ and satisfy
  \[
 | D^\alpha a_i (x)| \le C h_{X,\Omega}^{-|\alpha|}, \qquad
 x\in\Omega, \quad 1\le i\le N,
 \]
\hw{for $\alpha\in\N_0^d$ with $|\alpha|\le r$. Here, $C>0$ is a
  constant }independent of $X$, $x$ and $a_i$.
\end{lemma}
\begin{proof}
  This is proven in \cite{melenk2005approximation} and \cite{mirzaei2015analysis}.
\end{proof}

\section{Error estimates for quasi-interpolation} \label{sec: error_est}

We will now derive error estimates of the quasi-interpolation
process, if the target functions are from $C^k(\overline{\Omega})$. The
proofs are very similar to the proofs for Sobolev spaces in
\cite{melenk2005approximation, mirzaei2015analysis}.

\begin{remark}
In the following proof we require that the target function $f$ is defined
on the line connecting two arbitrary points $x, y\in\Omega$. This
can be achieved by either assuming that $\Omega$ 
is convex or that $f$ is defined on all of $\R^d$. \hw{For more 
  general domains $\Omega\subseteq\R^d$ it is possible to extend
  functions from  $\Omega$ to all of $\R^d$, provided $\Omega$ has a
  sufficiently smooth boundary, see for example
  \cite{Whitney-34-1}. Hence, for practical purposes it is no
  significant restriction to stick to convexity in the following theorem.}

  \end{remark}

Recall that $m\in\N_0$ denotes the degree of the polynomials which are
reproduced and that $r\in\N_0$ is the smoothness of the weights and
$k\in\N$ is the smoothness of the target function.

\begin{theorem}
  Let $\Omega\subseteq\R^d$ be a bounded, convex domain.
  Under all the above assumptions there is a constant $C>0$ such that 
  \[
  \|f-Q_X(f)\|_{C^\ell(\overline{\Omega})} \le C h_{X,\Omega}^{\min\{m+1,k\}-\ell}
  \|f\|_{C^{\min\{m+1,k\}}(\overline{\Omega})}
  \]
  for all $f\in C^k(\R^d)$ and all $0\le \ell \le \min\{r,m+1,k\}$.
\end{theorem}
\begin{proof}
  As $X$ is quasi-uniform, there is a constant $M>0$ independent of
  $N$ such that $\# I(x) \le M$ for all $x\in \Omega$. This
  follows as usual from a geometric argument. We have $j\in I(x)$ if
  and only if $\|x-x_j\|_2<\delta_j = c_\delta
  h_{X,\Omega}=:\delta$, i.e. if and only if $x_j\in
  \mathcal{B}(x,\delta)$. This implies 
  \[
  \bigcup_{j\in I(x)} \mathcal{B}(x_j,q_X) \subseteq \mathcal{B}(x,\delta+q_X),
  \]
  and as the balls on the left-hand side are disjoint, we can conclude
  \[
  \# I(x) q_X^d \le (\delta+q_X)^d.
  \]
  Quasi-uniformity $q_X\le h_{X,\Omega}\le c_q q_X$ then gives the
  upper bound
  \[
  \# I(x) \le \left(\frac{\delta}{q_X} + 1\right)^d \le
  \left(c_\delta c_q +1\right)^d=:M.
  \]

  Next, we use the notation $X(x)=\{x_j : j\in I(x)\}$. By the polyonmial reproduction of the moving least-squares process, we have for any $p\in\pi_m(\R^d)$, any $\widehat{x}\in
 \Omega$ and any $\alpha\in\N_0^d$ with $|\alpha|\le \min\{r,k\}$, 
 \begin{eqnarray*}
\lefteqn{   |D^\alpha f(\widehat{x})-D^\alpha Q_X(f)(\widehat{x})| = 
  | D^\alpha f(\widehat{x})-D^\alpha p(\widehat{x}) + D^\alpha p(\widehat{x}) -
  D^\alpha Q_X(f)(\widehat{x})|}   \\
  & \le & |D^\alpha f(\widehat{x})-D^\alpha p(\widehat{x})| + \sum_{j\in I(\widehat{x})}
  | D^\alpha a_j(\widehat{x})| | p(x_j)-f(x_j)|\\
  & \le & |D^\alpha f(\widehat{x}) -D^\alpha p(\widehat{x})| + 
  \|f - p\|_{\ell_\infty({X}(\widehat{x}))} \left(1+\sum_{j\in
   I(\widehat{x})} C_\alpha h_{X,\Omega}^{-|\alpha|}\right) \\
  &\le & |D^\alpha f(\widehat{x}) -D^\alpha p(\widehat{x})| +
  \| f - p\|_{\ell_\infty({X}(\widehat{x}))}  (1+C_\alpha M) 
  h_{X,\Omega}^{-|\alpha|}. 
 \end{eqnarray*}
Next, we choose $p$ as the Taylor polynomial
$T_{\min\{m,k-1\}} f$ of $f$ about $\widehat{x}$, i.e.
\[
p(x) = T_{\min\{m,k-1\}}f(x) = \sum_{|\bbeta|\le \min\{m,k-1\}} \frac{D^\bbeta
  f(\widehat{x})}{\bbeta!}(x-\widehat{x})^\bbeta.
\]
For $|\alpha|\le \min\{m,k-1\}$, it is well-known and shown by a straight-forward
calculation that 
\[ D^\alpha T_{\min\{m,k-1\}} f = T_{\min\{m,k-1\}-|\alpha|} D^\alpha f . \]
This immediately shows for
$|\alpha|\le \min\{m,k-1\}$ that $D^\alpha f(\widehat{x}) =
D^\alpha p (\widehat{x})$ holds. Moreover, as $\Omega$ is convex, 
we can use the Lagrange remainder formula in the form
\[
f(x) -  p(x) = 
\sum_{|\bbeta|=\min\{m+1,k\}}\frac{D^\bbeta f(\xi)}{\bbeta!} (x-\widehat{x})^\bbeta,
\]
with $\xi$ on the line connecting $x$ and $\widehat{x}$. This
shows for $x\in {X}(\widehat{x})$, the bound
\[
|f(x) - p(x)| \le C \delta^{\min\{m+1,k\}}
\|f\|_{C^{\min\{m+1,k\}}(\overline{\Omega})}. 
\]
Combining this with the estimate above, we find for $|\alpha|\le
\min\{r,m,k-1\}$,
\begin{equation}\label{est1}
\|D^\alpha f - D^\alpha Q_X(f)\|_{C(\overline{\Omega})} \le C
h_{X,\Omega}^{\min\{m+1,k\}-|\alpha|} \|f\|_{C^{\min\{m+1,k\}}(\overline{\Omega})}.
\end{equation}
We finally need to extend this bound to the situation of $|\alpha|\le
\min\{r,m+1,k\}$. Hence, we only have to show something in addition if
$\min\{m+1,k\}\le r$. In this case, we have for the  Taylor polynomial
$p=T_{\min\{m,k-1\}}$
and $|\alpha|=\min\{m+1,k\}=\min\{m,k-1\}+1$ obviously $D^\alpha
p(\widehat{x})=0$. This then shows
\[
|D^\alpha f(\widehat{x})-D^\alpha Q_X(f)(\widehat{x})| \le
|D^\alpha f(\widehat{x})| \le \|f\|_{C^{\min\{m+1,k\}}(\overline{\Omega})},
\]
which is the desired extension of (\ref{est1}).
\end{proof}

Obviously, it makes sense to align the three parameters $m,k,r$ as
follows.
\begin{corollary}\label{cor: quasi interpolation error}
  If the smoothnesses are chosen as $k=r=m+1$  then,
  \begin{equation}\label{eq: quasi-interpolation key}
        \|f-Q_X(f)\|_{C^\ell(\overline{\Omega})} \le C_{k,\ell}
  h_{X,\Omega}^{k-\ell}\|f\|_{C^k(\overline{\Omega})},
  \end{equation}
for all $f\in C^k(\overline{\Omega})$ and all $0\le \ell\le k$.
\end{corollary}

\section{Multiscale based on quasi-interpolation} \label{sec: multiscale}

This section introduces our multiscale approach via our error
correction transform. We begin by presenting the setup. Then, we
provide the algorithm and a first illustrative example. Finally, we
use the above error estimates to derive a convergence result for our
multiscale scheme based on quasi-interpolation approximation. \hw{As a
  matter of fact, this result is rather a stability than a convergence
  result, as we can only show that convergence of the multilevel
  scheme is comparable to that of the single-level scheme. Hence, this
  result should only be seen as a starting point in answering the
  following two essential questions. Can a non-converging
  quasi-interpolation scheme be turned into a converging multilevel
  scheme and can the convergence of an already converging scheme be
  improved?}

A multiscale approximation is an iterative approach that combines trends from different scales of the approximated functions. We start from the approximation of low-resolution samples of the function. Then, each consecutive iteration approximates the error of the previous iteration with samples of higher resolution and uses the result to improve the previous approximant. 

\begin{figure}
    \centering
    	\begin{subfigure}[b]{.25\textwidth}
		\centering
		\includegraphics[width=\textwidth]{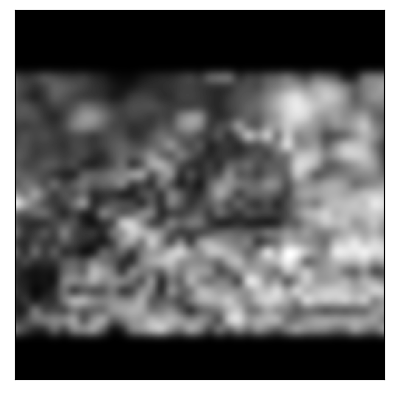}
	\end{subfigure} 
	    	\begin{subfigure}[b]{.25\textwidth}
		\centering
		\includegraphics[width=\textwidth]{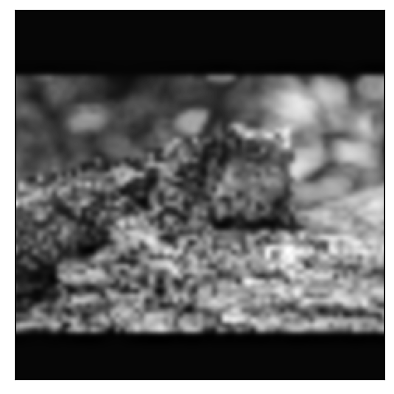}
	\end{subfigure} 
	    	\begin{subfigure}[b]{.25\textwidth}
		\centering
		\includegraphics[width=\textwidth]{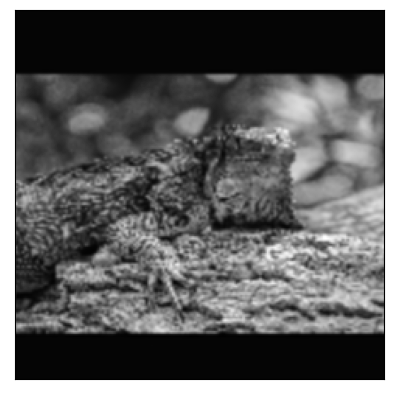}
	\end{subfigure}  \\
	    	\begin{subfigure}[b]{.25\textwidth}
		\centering
		\includegraphics[width=\textwidth]{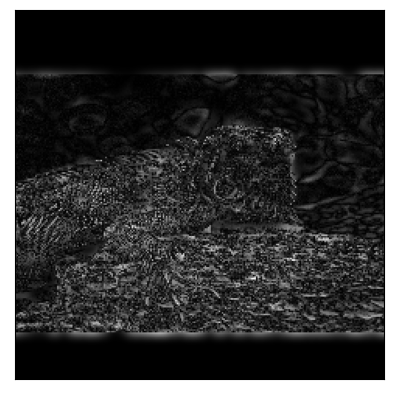}
	\end{subfigure} 
	    	\begin{subfigure}[b]{.25\textwidth}
		\centering
		\includegraphics[width=\textwidth]{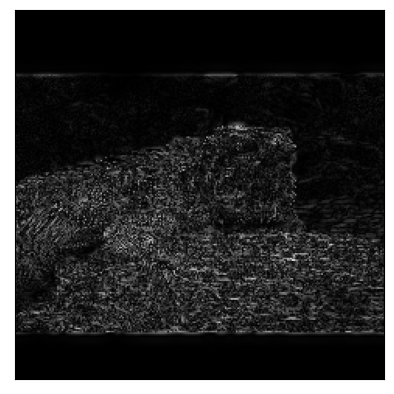}
	\end{subfigure} 
 \caption{An illustrative example of multiscale approximation of an
   iguana photo~\cite{photo}. The upper row is the
   approximation in different levels. The lower row is the error of
   the previous iteration. To calculate the next iteration, we approximate the error and add it to the previous approximation. We use a quasi-interpolation of the form~\eqref{eq: general_scheme} with $\sum_i a_i = 1$ (Shepard's method), scaling factor of $\mu=0.5$, and the sample fill distance is $h_X = 0.008$. For more details see Section~\ref{sec: numerical}. } 
    \label{fig: multiscale demo}
\end{figure}

Here, the general set-up is as follows. We have a sequence 
of data sets $X_1, X_2,\ldots\subseteq\Omega$ with mesh norms $h_j=h_{X_j,\Omega}$,
which are \hw{strictly} monotonically decreasing. To ensure a certain uniformity in
decrease, we will assume that $h_{j+1}= \mu h_j$ for some fixed \textit{scaling factor}
$\mu\in (0,1)$.  We will also require that the
sequence is {\em quasi-uniform}, which means that there is a constant
$c_q$ such that, with  $q_j=q_{X_j}$, 
\[
q_j \le h_j \le c_q q_j.
\]
Next, \hw{as the data sets are quasi-uniform, we will, for a fixed
  level $j$,  use the same support radius $\delta_j$ for all the weight
  functions $w_i^{(j)}$ of the moving-least squares operator on that
  level $j$. To be more precise, we pick}
 $\delta_j=\nu  h_j$ with a fixed $\nu>1$ and form moving
  least-squares approximation operators for level $j$ based on the
  points in  $X_j=\{x_1^{(j)},\ldots,x_{N_j}^{(j)}\}$ as 
\[
Q_jf:= \sum_{i=1}^{N_j} f(x_i^{(j)}) a_i^{(j)},
\]
where the $a_i^{(j)}$ are built using the weights
\[
w_i^{(j)}(x)= \Phi\left(\frac{x-x_i^{(j)}}{\delta_j}\right),
\]
with a compactly supported function $\Phi \colon \R^d\to\R$ as before.

The multiscale quasi-interpolation algorithm then works as a standard
residual correction algorithm, as detailed in Algorithm \ref{alg1}. We provide in Figure~\ref{fig: multiscale demo} an illustrative example of applying our residual correction algorithm for the approximation of an image. 

\begin{algorithm}
	\caption{Multiscale quasi-interpolation \label{alg1}}
	\hspace*{\algorithmicindent} \textbf{Input:} function to approximate $f$, levels number $n$, sequence of data sets $X_1,
	\ldots X_{n} \subseteq\Omega$ \\
	\hspace*{\algorithmicindent} \textbf{Output:} The approximation of $f$, that is $f_n$ 
\begin{algorithmic}
\STATE{$f_0=0$, $e_0=f$}
\FOR {$j=1,2,\ldots,n$}
\STATE{Compute the quasi-interpolation $s_j=Q_j e_{j-1}$ to $e_{j-1}$ based on data over $X_{j}$ }
\STATE{$f_j = f_{j-1}+s_j$}
\STATE{$e_j = e_{j-1}-s_j$}
\ENDFOR
\end{algorithmic}
\end{algorithm}

With this at hand, we have the following convergence result for
our multiscale quasi-interpolation approximation.
\begin{theorem} 
\label{thm:multiscale}
Let $\Omega\subseteq\R^d$ be a convex, bounded domain. Let $X_j\subseteq\Omega$ be a sequence of quasi-uniform data
  sets with fill-distances $h_{j+1}=\mu h_j$, $\mu \in (0,1)$. Let
  $\delta_j=\nu h_j$ with $\nu>1$. Let $\Phi\in
  C^k(\overline{\Omega})$, $k\in\N$,
  and let the moving quasi-interpolation approximation reproduce polynomials
  of degree $m\ge k-1$. \hw{Then, the approximation $f_n$ from Algorithm
  \ref{alg1} satisfies the bound}
\begin{equation}\label{eq:multiscale-step}
  \|f-f_n\|_{C(\overline{\Omega})} \le \hw{\frac{C_{0,k}h_0^k}{C_{k,k}}} (C_{k,k}\mu^k)^n
  \|f\|_{C^k(\overline{\Omega})},  
\end{equation}
  where the constants are from Corollary \ref{cor: quasi interpolation error}. Hence, the method
  converges if we choose $\mu$ such that $C_{k,k}\mu^k <1$.
  \end{theorem}
\begin{proof}
  By Corollary~\ref{cor: quasi interpolation error}, we have with $\ell=k$, 
  \[
  \|e_j\|_{C^k(\overline{\Omega})} =\|e_{j-1}-Q_j
  e_{j-1}\|_{C^k(\overline{\Omega})} \le C_{k,k}
  \|e_{j-1}\|_{C^k(\overline{\Omega})}, \qquad 1\le j\le n.
  \]
  Iterating this leads to the bound
  \[
  \|e_n\|_{C^k(\overline{\Omega})} \le C_{k,k}^n
  \|e_0\|_{C^k(\overline{\Omega})} = C_{k,k}^n
  \|f\|_{C^k(\overline{\Omega})}.
  \]
  Applying  Corollary \ref{cor: quasi interpolation error} with $\ell=0$ then yields
  \begin{eqnarray*}
\|e_n\|_{C(\overline{\Omega})} &=& \|Q_{n} (e_{n-1}) - e_{n-1}\|_{\hw{C(\overline{\Omega})}} \le C_{0,k} h_{n}^k
  \|e_{n-1}\|_{C^k(\overline{\Omega})} 
  \\
  &\le & C_{0,k} h_n^k C_{k,k}^{\hw{n-1}}
  \|f\|_{C^k(\overline{\Omega})} \le  \hw{\frac{C_{0,k}h_0^k}{C_{k,k}}}
  (\mu^kC_{k,k})^n \|f\|_{C^k(\overline{\Omega})}, 
  \end{eqnarray*}
  which is the stated result.
\end{proof}
Note that the convergence enforcing term $(C_{k,k}\mu^k)^n$ contains
in particular $\mu^{kn} = h_n^k$ which coincides with the convergence
enforcing term for the standard quasi-interpolation scheme. Here,
however, we also have the additional factor $C_{k,k}^n$. If
$C_{k,k}>1$ this would reduce the convergence significantly. If,
however, $C_{k,k}<1$, we would have improved
convergence. Unfortunately, the constant $C_{k,k}$ is not really
accessible and also depends on the factor $\nu>1$ from $\delta_j=\nu
h_j$. Though, it might be possible to enfore a small $C_{k,k}$ by
choosing $\nu$ wisely, a simple calculation in the case $k=1$ shows
that the above mathematical techniques always lead to a constant
$C_{1,1}>1$ which can become arbitrarily close to $1$ by increasing
$\nu$. Nonetheless, the numerical examples below show that these
mathematical techniques do not truly reflect the error behaviour of
the multilevel quasi-interpolation scheme. Obviously, more research in this
direction is necessary.

\section{Transition to the setting of manifold-valued data } \label{sec: transition_to_manifold}

In this section we construct the theory for quasi-interpolation and
multiscale approximation for manifold-valued functions. We begin with changing
notations for manifolds. Then, we show how the moving least squares
method is translated to manifold-valued data via the Karcher
mean. After that, we define the smoothness descriptor which is the
manifold version of $C^k$. We end with the quasi-interpolation and
multiscale versions and results that we do not prove. 

\subsection{Updating our notation}

Let $\cM$ be a Riemannian manifold equipped with its Riemannian metric tensor $g$ which induces an intrinsic norm $\abs{\cdot}_{g(p)}$ on the tangent space $T_p\cM$ at $p\in\cM$. Then, the Riemannian geodesic distance $\rho(\cdot,\cdot)\colon\cM^2\to\mathbb{R}_+$ is
\begin{align}~\label{RiemannianDistance}
	\rho(b,q)=\inf_{\Gamma}\int_{0}^{1}\abs{\dot{\Gamma}(t)}_{g(\Gamma(t))}dt, 
\end{align}
where $\Gamma\colon[0,1]\to\cM$ is a curve connecting points $\Gamma(0)=b$ and $\Gamma(1)=q$.

Manifold-valued functions map a Euclidean domain $\Omega$ to a fixed, known manifold $\cM$. We denote manifold-valued functions with uppercase letters and functions with vector values, such as values from a tangent space of $\cM$, in lowercase letters. Also, denote the exponential map of a vector $v$ in the tangent space $T_{ \cM }(b)$ around a base point $b$, and its inverse logarithm map by 
\[
b \oplus v  = \Exp{b}{v}, \quad \text{ and } \quad q \ominus b  = \Log{b}{q} , \quad b,q \in \cM .
\]
Namely, $b \oplus v$ maps a vector $v$ at $T_{ \cM }(b)$ to the point on the manifold which lies at a distance $\norm{v}$ along the geodesic line from $b$ parallel to $v$. If $q$ is inside the injectivity radius, then $b \oplus (q \ominus b) = q$ and the geodesic distance $\rho(b,q)$ is merely $\norm{q \ominus b}$, with the Euclidean norm at $T_{ \cM }(b)$. When $G$ and $H$ are $\cM$-valued functions (or sequences),  we assume a common representation (or indexing), for example, $G \ominus H = G(x) \ominus H(x)$ for all $x$ in the common domain. 

\subsection{Applying Karcher means}  

We generalize the quasi-interpolation scheme~\eqref{eq: general_scheme} to manifold-valued scattered data, based on the given Riemannian structure, via the notion of Karcher mean. This intrinsic averaging method adapts scalar operators to manifold values and naturally agrees with the quasi-interpolation MLS formulation~\eqref{eq: moving least squares} as we will see next.

We refer to the linear combination at~\eqref{eq: general_scheme} and
assume that the coefficients form a partition of unity, that is they satisfy
\[
\sum_{i=1}^N a_i(x) = 1, \quad x \in \Omega .
\]
This condition implies the exact reproduction of constants, see
Definition~\ref{def:poly_rep}. \rev{Note that we do not assume the positivity of the coefficients, which is usually not guaranteed}. More details on such operators are given on Section~\ref{sec: numerical}. Next, for any $x$, the value
in~\eqref{eq: general_scheme} can be characterized as the unique
solution of the optimization problem 
\begin{align}~\label{LinearFrechetMean}
 \arg\min_{y\in\mathbb{R}} \sum_{i=1}^N a_i \|y-f(x_{i})\|^2.
\end{align}
Namely, this form identifies the linear combination~\eqref{eq:
  general_scheme} as the Euclidean center of mass. In particular, the
result of~\eqref{LinearFrechetMean} is the center of mass of $\{
f(x_{i})\}_{i=1}^N$ with respect to the weights $\{ a_{i}\}_{i=1}^N$. 

For $\cM$-valued data, we rewrite the optimization
problem~\eqref{LinearFrechetMean} by replacing the Euclidean distance
with the Riemannian geodesic distance~\eqref{RiemannianDistance} and
restric the search space to our manifold. Namely, for the samples of a
manifold-valued function $F$ we define, 
\begin{align}~\label{eq: quasi-interpolation manifold}
 Q^\mathcal{M}_X(F)(x)=\text{arg}\min_{q \in \cM} \sum_{i=1}^N a_i \, \rho(q,F(x_i))^2 , \qquad x \in \Omega.
\end{align}
The above is the manifold-valued counterpart of the approximation
operator $Q$ from \eqref{eq: general_scheme} for constant reproduction. We focus on cases when the solution of~\eqref{eq: quasi-interpolation manifold} exists uniquely, and then it is termed the \emph{Riemannian center of mass}~\cite{grove1973conjugatec}. \rev{This solution is also known as the  Karcher mean, especially concerning matrix spaces. In general metric spaces, however, the same solution is called the Fr\'{e}chet mean}, see~\cite{Karcher2014riemannian}. 

\rev{The question of the existence and uniqueness of the Riemannian center of mass was the subject of many studies. We briefly survey it while highlighting some of them. The global well-definedness of~\eqref{eq: quasi-interpolation manifold} with nonnegative coefficients $a_i\geq 0$ is studied in~\cite{kobayashi1963foundations}. More on a globally unique solution is derived when the manifold $\cM$ has a nonpositive sectional curvature, e.g.,~\cite{hardering2015intrinsic, karcher1977riemannian, sander2016geodesic}. On the other hand, recent studies regarding manifolds with positive sectional curvature show the necessary conditions for uniqueness on the spread of points with respect to the injectivity radius of $\cM$~\cite{dyer2016barycentric, SvenjaWallnerConvergenceSphere}.}

\rev{Even with a guarantee regarding the existence and uniqueness of the Riemmanian center of mass, typically, there is no close form of it, particularly when there are more than two points involved (two points average is equivalent to the center point of their connecting geodesic). Therefore, calculating the Riemannian center of mass is done numerically with iterative processes, see, e.g.,~\cite{iannazzo2019derivative}. Thus, one interpretation of many alternative methods for adapting approximation operators to manifold values is seeing them as finite approximations for~\eqref{eq: quasi-interpolation manifold}.} Finite procedures include, for example, exp-log methods~\cite{grohs2012definability, rahman2005multiscale}, repeated binary averaging~\cite{dyn2017global, wallner2005convergence}, inductive means~\cite{dyn2017manifold}, and more. Next, in the numerical section, we show examples of algorithms for calculating the Karcher mean.

\subsection{Smoothness descriptor}

For obtaining error bounds, we must first show how we measure the
smoothness of a manifold-valued function. In other words, we seek
manifold-valued settings analogous to the spaces of $k$-times
continuously differentiable functions and their norms. One natural
generalization uses smoothness descriptors~\cite{grohs2013quasi,
  grohs2017scattered}. Specifically, we are inspired
by~\cite{grohs2017scattered} where the smoothness descriptor uses the
covariant derivatives of a manifold-valued function $F:\Omega\to
\cM$. In our analysis, we examine the deviation between $F$ and its
approximation $ Q^\mathcal{M}_X(F)$, that is $F\ominus
Q^\mathcal{M}_X(F)$. This quantity is a vector field attached to
$Q^\mathcal{M}_X(F)$ and therefore, the smoothness descriptor we
present is given as the $C^k$ norm of a vector field. 

Let $G:\Omega\to \cM$ be a function with a vector field $W$ attached
to it. That is, for all $x\in\Omega$, $W(x)\in T_{G(x)}\cM$. Then, the
$r$-th index of the covariant derivative of $W$ in the direction $x^l$ \rev{(the $l$th component of $x$)} is given by
\begin{equation}\label{eq: covariant derivative}
    \frac{D}{dx^l}W^r(x) := \frac{dW^r}{dx^l}(x) + \Gamma^r_{ij}(G(x)) \frac{d G^i}{d x^l}W^j(x),
\end{equation}
where we sum over repeated indices and denote with $\Gamma^r_{ij}$ the
Christoffel symbols associated to the metric of $\cM$ as in
\cite{do1992riemannian}. In addition, $W^i$ and $G^i$ mean the $i$-th component
 of $W$ and $G$, respectively. 

For iterated covariant derivatives we introduce the
symbol $\mathcal{D}^\mathbf{l}(F\ominus G)$ which means covariant partial differentiation along $F\ominus G$ with respect to the
multi-index $\mathbf{l}$ in the sense that
\[
\mathcal{D}^\mathbf{l} (F\ominus G) := \frac{D}{dx^{l_k}}\cdots\frac{D}{dx^{l_1}}(F\ominus G). 
\]
All derivatives are vector fields attached to $G$ \rev{and the multi-index here must respect order as the covariant partial derivatives do not commute.} Finally, we define the vector field smoothness descriptor of the error $F\ominus G$, which is the manifold analogue of~\eqref{eqn:smoothness_desc_real_values}.
\begin{definition}
For functions $F, G \colon  \Omega\to\cM$, and $U\subseteq\Omega$, we define \rev{their mutual} $k$-th order smoothness descriptor:
  \begin{equation*}
      \norm{F\ominus G}_{C^k(\overline{U})}:= 
      \max_{\abs{\mathbf{l}} \le k}  \,
      \sup_{x\in U}{\abs{\mathcal{D}^{\mathbf{l}}(F\ominus G)}_{g(G(x))}},
  \end{equation*}
\rev{where $g$ is the metric tensor of $\cM$.}
\end{definition}

\rev{For the following discussion, we denote by $B \colon \Omega\to\cM$ a reference function such that $B(x)$ is within the injectivity radius of $F(x)$ for every $x\in\Omega$. We later consider some practical aspects of constructing such a reference. Next,} we further define the induced smoothness class:
\begin{definition}
We say that a function $F\colon \Omega\to\cM$ is locally in \rev{$C^k(p,B,\overline{\Omega})$}, with respect to a fixed point $p\in \cM$ \rev{and a reference function $B \colon \Omega\to\cM$,} if $\norm{F\ominus B}_{C^k(\overline{U})}$ is finite inside the ball around $p$ within the injectivity radius. If, in addition, $F \in C^k(p,\overline{\Omega})$ for any $p \in \cM$, we say that $F \in C^k(\overline{\Omega})$ \rev{with respect to $B$}.
\end{definition}

Given a function $F\in C^k(\overline{\Omega})$, we wish for a manifold analogue of~\eqref{eq: quasi-interpolation key}, that is, a bound of the form
\begin{equation}\label{eq: desired quasi-interpolation}
    \norm{F \ominus Q^\cM(F)}_{C^\ell(\overline{\Omega})} \leq C_{k,\ell}  h_{X,\Omega}^{k-\ell} \norm{F\ominus B}_{C^k(\overline{\Omega})},
\end{equation}
where $h_{X,\Omega}$ is the fill distance,
see~\eqref{eqn:fill_dist_separation_rad}. In particular, we should give a special care for the case of $\ell=0$, that is $
\norm{F \ominus Q^\cM(F)}_\infty := \norm{F \ominus
  Q^\cM(F)}_{C^0(\overline{\Omega})}$.

In the following remark, we survey two related bounds.
\begin{remark}
Two known results in the spirit of~\eqref{eq: desired quasi-interpolation} are worth mentioning.
The first, from~\cite{grohs2013quasi}, uses embedding of $F$ in a higher dimensional Euclidean space. It is shown that the quasi-interpolation in the manifold is converging to the linear quasi-interpolation of the Euclidean space. Then, for $F\in C^\alpha$, $l < \alpha$, and for any chart $\gamma$, it was shown that
\[
\left(\frac{d}{dx}\right)^l\left(\gamma\circ F - \gamma \circ Q^\cM(F)\right) = \mathcal{O}(h^{\alpha-l}).
\]
The second result is from \cite{grohs2017scattered} and develops the Taylor expansion of $G(x, y)=F(y)\ominus Q^\cM(F)(x)$, which is in the tangent space of $Q^\cM(F)(x)$. It then leads to a similar result with
\[
\rho(F, Q^\cM(F)) \leq C_Q \Theta_{k}(F) \sup_{1\leq r\leq k}\ \sup_{x,y\in\Omega}{\nabla^r_2{F(y)\ominus Q^\cM(F)(x)}} h^k , 
\]
where $\nabla^r$ is the covariant derivative of a bivariate function \rev{and $ \Theta_{k}(F)$ is a smoothness descriptor similar to $\norm{F}_{C^k}$ but measured on a vector fields with tangent vectors from the values of $F$}. Nevertheless, the multiscale proof assumes a stronger bound of the form~\eqref{eq: desired quasi-interpolation}.
\end{remark}

\subsection{Multiscaling of manifold-valued data} \label{sec:multiscal_manifolds}

As in the Euclidean case, we will use the quasi-interpolation operators in a multiscale fashion to derive higher approximation capabilities and a dynamic framework for analyzing and processing data. We will now describe the scheme in more detail. First, we define the manifold scattered data version of quasi-interpolation and multiscale approximations. Then, we describe the assumptions needed and the resulted multiscale error bound. 

Assume $X_1, X_2, \ldots, X_n \subseteq \Omega$  is a sequence of discrete sets with more and more data sites. We choose support radii $\delta_1>\delta_2>\cdots>\delta_n$, which should be proportional to the associated  fill distances $h_1>h_2>\cdots>h_n$ of the sets $X_j$. Then, for $j=1,2,\ldots,n$ we compute local quasi-interpolants according to~\eqref{eq: quasi-interpolation manifold}, that is
\begin{equation} \label{eqn:localManifoldQuasi}
Q^{\cM}\left(X_j, F \right)(x):= \argmin_{q \in
  \cM} \sum_{\{i \mid x_i \in X_j \}} \rev{a_i \ \rho}(q,F(x_{i}))^2 ,
\end{equation}
However, as explained above, we will not apply these operators to the function itself but rather to the error function of the previous iteration. In the manifold setting, this requires some more preparation. 

Here, the multiscale method starts by setting a base manifold-valued function $B$. In algebraic settings, where $\cM$ is a Lie group, we can fix $B$ as the constant function with the identity element. Otherwise, we wish this function to be in the vicinity of the data. \rev{The reason is that we wish to be able to parallel transport the error vectors to the tangent space of $B$. Therefore, by $ B \oplus e_{j-1}$, we mean that each vector $ e_{j-1}(x)$ is parallel transported to the tangent of $B(x)$ and similar each vector in $S_j \ominus B$ is parallel transported to the tangent of $F_{j-1}$. Note that since Lie groups are parallelizable, this action is always possible. As a construction for $B$,} we can use, for example, a Voronoi tessellation with the samples $\rest{F}{X_{1}}$. We set at the begining $F_0 = B$ and $e_0 = F \ominus B$. \rev{The whole algorithm is described in detail in the following Algorithm~\ref{alg2}.}

\begin{algorithm}
	\caption{\rev{Manifold-valued multiscale quasi-interpolation} \label{alg2}}
	\hspace*{\algorithmicindent} \textbf{Input:} function to approximate $F$, levels number $n$, sequence of data sets $X_1\ldots X_{n} \subseteq\Omega$, a base manifold function $B$ \\
	\hspace*{\algorithmicindent} \textbf{Output:} The approximation $F_n$ 
	\begin{algorithmic}
		\STATE{$F_0=B$}
		\STATE{$e_0 = F \ominus B$}
		\FOR {$j=1,2,\ldots,n$}
		\STATE{$S_j = Q^{\cM}\left( X_{j},  B \oplus e_{j-1}  \right)$}
		\STATE{$F_j =  F_{j-1} \oplus (S_j \ominus B )$}
		\STATE{$e_j = F \ominus F_j$}
		\ENDFOR
	\end{algorithmic}
\end{algorithm}

\rev{Note that in the Euclidean case, Algorithm~\ref{alg2} is reduced to the previously scalar-values multiscale residual correction in Algorithm~\ref{alg1},} as $+$, $-$ naturally replace $\oplus$ and $\ominus$, respectively, and $B \equiv 0$ (the identity there). 

Our next result implies that with a sufficiently accurate quasi-interpolation operator, the multiscale approximation error decays as we observed in the scalar case. 
\begin{corollary} \label{cor:manifold_quasi_err}
  Let $\cM$ be a complete, open manifold. Then, under the condition of Theorem~\ref{thm:multiscale}, its notation, and assumption~\eqref{eq: desired quasi-interpolation}, the following holds for all $F \in C^k(\overline{\Omega})$:
    \begin{equation}\label{eq:multiscale-step-manifolds}
  \norm{F\ominus F_n}_\infty \le C_{0,k} (C_{k,k}\mu^k)^n
  \norm{F\ominus B}_{C^k}.
\end{equation}
\end{corollary}
\rev{Note that the error bound in Corollary~\ref{cor:manifold_quasi_err} is theoretical and does not reflect any numerical errors that may arise, for example, from calculating~\eqref{eqn:localManifoldQuasi}. Next, we address different aspects of the numerical evaluation of our method.}  
\section{Numerical examples} \label{sec: numerical}

This section provides a numerical perspective to our multiscale
study. First, we present some demonstrations of the multiscale scheme
and then illustrate the numerical aspects of some of the theoretical
findings we obtained. Finally we suggest an application to the scheme
in anomaly detection. The entire source code is available as python
implementation for reproducibility at the repository \cite{rep}. 

With $\phi_{d,k}$ we denote the compactly supported radial basis
functions of minimal degree from \cite{wendland_2004} with indices $d,
k$ meaning that the RBF $\Phi=\phi_{d,k}(\|\cdot\|_2)$ is in $C^{2k}(\R^d)$ and is
\textit{positive definite} on $\R^{\ell}$ for $\ell\leq d$.

In what follows, we use the Wendland function $\phi_{3,1}$, see \cite{wendland1995piecewise}, to construct the quasi-interpolation kernel,
\begin{equation}
\label{eq: wendland function}
   \phi(r):= \phi_{3,1}(r) = (1-r)^4_+(4r+1).
\end{equation}
To construct $Q_X$, we use $\phi$ as in Shepard's method, that is
\begin{equation}
    \label{eq: shepard}
    Q_X(f)(x) = \sum_{x_i \in
      X}{\frac{\phi(\|x-x_i\|_2)f(x_i)}{\sum_{x_j\in
          X}{\phi(\|x-x_j\|_2)}}}, \qquad x \in \Omega. 
\end{equation}
This method, by definition, reproduces constants, that is $Q_X(c)(x)=c$. 
We use $Q_X$ for our multiscale approximation, but also as a reference
approximation method. In the examples below, we refer to the
quasi-interpolation approximation as ``single scale''. 

To evaluate error between a function $f$ and its approximation at the
$j$-th scale $f_j$, we use a discrete $L_\infty$ norm of the form: 
\begin{equation} \label{eqn:error}
    \max_{x_i \in G(R, h)}{\left|f(x_i) - f_j(x_i)\right|},
\end{equation}
where $G(R, h)$ is a grid of points in the subdomain $R \subset \mathbb{R}^d$ with a fill distance $h$. Here, we use a rectangular $R\subset \mathbb{R}^2$ strictly inside $\Omega$ \rev{where the approximation support is still inside the domain. In other words, we do not consider special boundary treatments.} Also, the resolution of the error evaluation is with $h = 0.02$, which is much finer than the sampling fill distance $h_X$. 

\rev{In the following examples, when we compare the multiscale with the single scale method, we use the same data samples having the same fill distance.} All calculations are done in \rev{P}ython using packages from~\cite{2020NumPy-Array, 2020SciPy-NMeth}, and for plots, we use~\cite{hunter2007matplotlib}.

\subsection{A comparison between multiscale and single scale: the real-valued case} \label{sec:first_comparison}

\rev{As our first test case, we demonstrate the multiscale approximation, specifically its error rates, compare with the quasi-interpolation.} We approximate a real function over $\R^2$, and use samples over the domain $\Omega := [-0.95, 0.95]^2$ with a fill distance $h_X = 0.375$ and a scaling factor $\mu = 0.8$. The sample sites, for each level, are generated from a Halton sequence~\cite{halton1960efficiency}. For more details, see Appendix~\ref{app: scattered data}.

We have approximated the following test function:
\begin{equation}\label{eq: approximated functions_h}
    h(x, y) = 5e^{-x^2-y^2}.
\end{equation}
\rev{We compare the multiscale method and the single scale method over the test function and depict the results in Figure~\ref{fig:numbers comparison}. The figure presents the error decay, measured according to~\eqref{eqn:error}. This first example} shows that the error of the multiscale approximation decays faster than the error of the quasi-interpolation with the same data. \rev{Note that this figure and some of the next error figures are given in a log-log scale. Namely, we apply the logarithm both over the fill distance and the observed error. This way, one can identify the slope of the error rate with the approximation order, that is, the power $k$ in~\eqref{eq:multiscale-step}. More on the usage of the log-log scale appears in the next section.}


\begin{figure}[!htb]
	\centering
	\includegraphics[width=.4\textwidth]{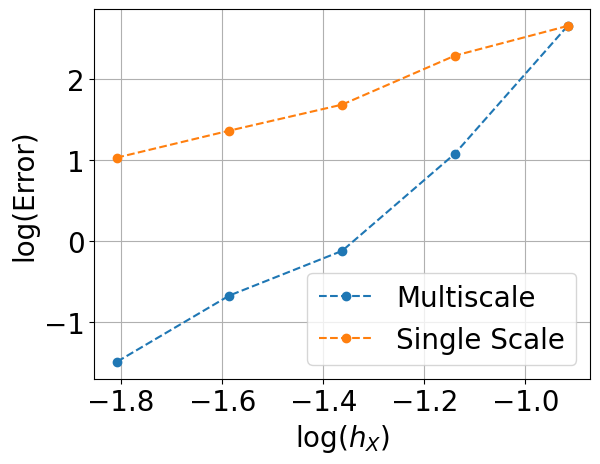}
	\caption{Error rates for~$h$ of~\eqref{eq: approximated functions_h}: the multiscale approximation versus
          the single scale approximation. The error is measured by
          maximum deviation error, see~\eqref{eqn:error}, and given in a
          log-log plot as a function of the fill distance,
          see~\eqref{eqn:fill_dist_separation_rad}.} 
	\label{fig:numbers comparison}
\end{figure}

\subsection{Validating our analysis}


In this part, we numerically validate several of our theoretical findings. The function we approximate is: 
\begin{equation}
\label{eq: approximation func}
    f(x, y) = \sin(2x + 1) \cos(3y + 1.5).
\end{equation}
Here, we continue to use $\phi_{3,1}$ as above and the data sites $X$
are grids with the same fill distance and size as the sampling grids before.
 
We investigate two of the main parameters of the multiscale method: the number of iterations, which we denote by $j$, and the scaling factor $\mu$. In particular, we test numerically Theorem~\ref{thm:multiscale}, by observing the error decay as a function of the iterations $j$ via a log simplification of~\eqref{eq:multiscale-step}:
\begin{equation}
    \label{eq: multiscale single mu fit}
    \log\norm{f_j -f}_{\infty}\leq \log{C_{0, k}\norm{f}_{C^k\left(\overline{\Omega}\right)}}+j\log{\left(C_{k,k}\mu^k\right)}.
\end{equation}
Here, since we use constant reproduction quasi-interpolation, we have $k=1$.

First, we set three fixed values of $\mu$ and apply our multiscale approximation for each one. Then, we calculate the slopes of each experiment when fitting a linear connection between the log of the approximation error and $j$, as in~\eqref{eq: multiscale single mu fit}. The results are depicted in the upper row of Figure~\ref{fig: multi mu fig} and indicate that the relation is indeed close to be linear as in the bound we obtained.

\begin{figure}[!ht]
    \centering
    \begin{subfigure}{0.3\textwidth}
        \includegraphics[width=\linewidth]{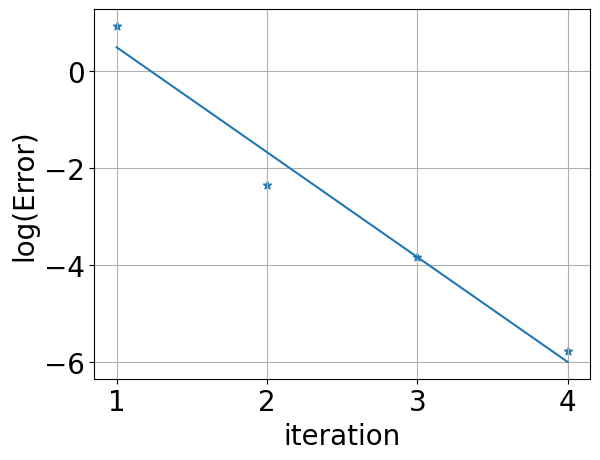}
        \caption{$\mu = 0.5$}
    \end{subfigure}
    \begin{subfigure}{0.3\textwidth}
        \includegraphics[width=\linewidth]{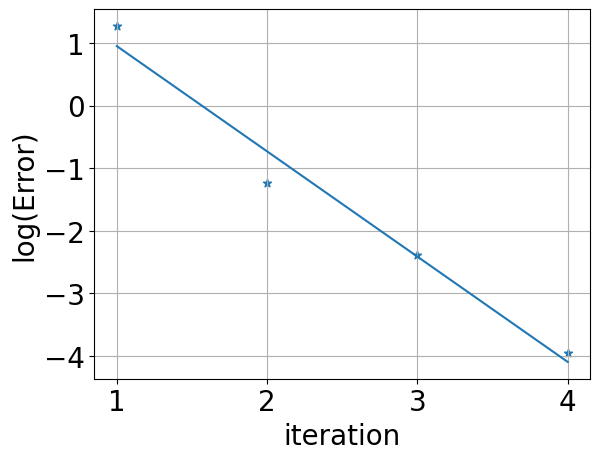}
        \caption{$\mu = 0.6$}
    \end{subfigure}
    \begin{subfigure}{0.3\textwidth}
        \includegraphics[width=\linewidth]{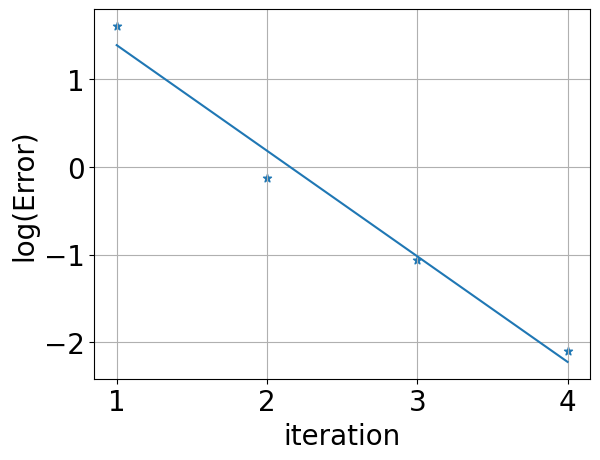}
        \caption{$\mu = 0.7$}
    \end{subfigure} \\
    \includegraphics[width=.4\linewidth]{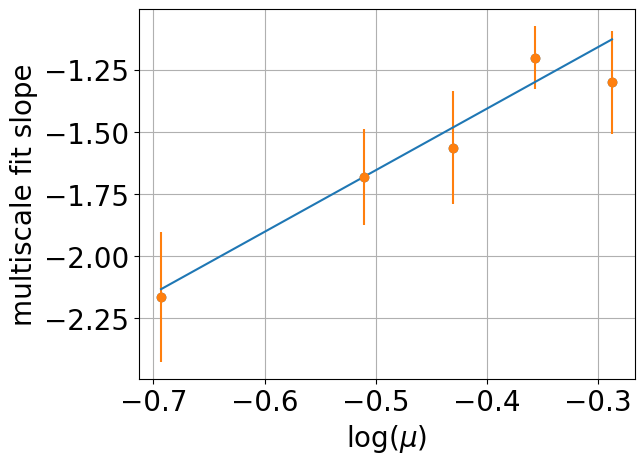}
    \caption{Exploring Theorem~\ref{thm:multiscale} via~\eqref{eq: multiscale single mu fit} as reflects by numerical tests: In the upper row, we show the approximation error as a function of the iterations and assign a slope value according to a linear fit. The bottom figure presents the connection between the slopes of each iteration and $\log(\mu)$ as in~\eqref{eq: multiscale milti mu fit}. The error bars are calculated from the different $\mu$-s fits. The error of each slope is the standard deviation derived from the fitted parameters' co-variance matrix. The result reveals the values $\log {C_{1,1}} = -0.42$ and $k = 2.47$. The latter is a larger value that the minimal guaranteed value of $1$.}
    \label{fig: multi mu fig}
\end{figure}

In~\eqref{eq: multiscale single mu fit}, the slope of the above linear connection is
\begin{equation} \label{eq: multiscale milti mu fit}
	\log{C_{k,k}} + k\log{\mu}.
\end{equation}
Recall that in order to guarantee convergence we must have $C_{k,k}\mu^k<1$. Therefore, to estimate this crucial quantity, we reuse the above results and for each $j$ we draw the average slope as a function of $\log(\mu)$. Finally, we use a linear regression to extract from the slopes and~\eqref{eq: multiscale milti mu fit} the parameters $\log {C_{1,1}}$ and $k$. The bottom plot of Figure~\ref{fig: multi mu fig} shows this linear fit, where we get $\log {C_{1,1}} = -0.42$ and thus $C_{1,1} = 0.66$. This value promises convergence for the $\mu$ values we used for the test. In addition, we have $k=2.47$, which is much better than theoretical bound of $1$.

\subsection{Obtaining quadratic rates with constant reproducing via multiscale}
\label{sec: complexity comparison}

The advantage of Shepard's method with constant reproduction~\eqref{eq: shepard} lies in its simplicity, which results in an attractive complexity and preferred execution time compared to higher degree reproduction schemes, see Remark~\ref{rem:QIcomplexity} below. In this section, we show that our multiscale approach, when using Shepard's operators, can provide a comparable approximation result to a quasi-interpolation with a higher polynomial reproduction degree. Namely, one can achieve much-improved error decay by using multiscale while benefiting from the low computation load of Shepard's method.

For this example, we approximate the function
	\begin{equation}\label{eq: approximated functions}
		g(x, y) = \sin(4x)\cos(5y),
	\end{equation}
 and use the same approximation settings as in Section~\ref{sec:first_comparison}. To obtain higher order polynomial reproduction, in this case, quadratic, we use the moving least-squares method from~\cite{wendland2001local}. 

In Figure~\ref{fig: const_multi_vs_quadratic_single}, we compare the multiscale method based on Shepard's method, a quasi-interpolation with constant reproduction, and two other quasi-interpolation methods, applied in a single scale fashion. One operator has a constant reproduction, and the other with quadratic reproduction. We present both the error decay rates and the timing comparison, measured in seconds on a standard laptop device. 

The error rates figure shows that the multiscale method, based on reproducing constants, has at least a similar decay rate as the quasi-interpolation, based on reproducing quadratic. These error rates are superior to those directly applying Shepard's method, the single scale with constant reproduction. The timing figure shows the calculating time, in seconds, needed for each method for evaluating the approximations on a fixed fine grid $G(R,h)$. The first phenomenon observed in the figure is the distinct difference between the two quasi-interpolations; see also Remark~\ref{rem:QIcomplexity}. Moreover, the figure also indicates that the multiscale shares the preferred timing of Shepard's method for small fill distance values. As the number of levels increases, the running time of the multiscale method gradually grows. 

\begin{figure}[!ht]
		\centering
		\begin{subfigure}{0.45\textwidth}
    \includegraphics[width=\textwidth]{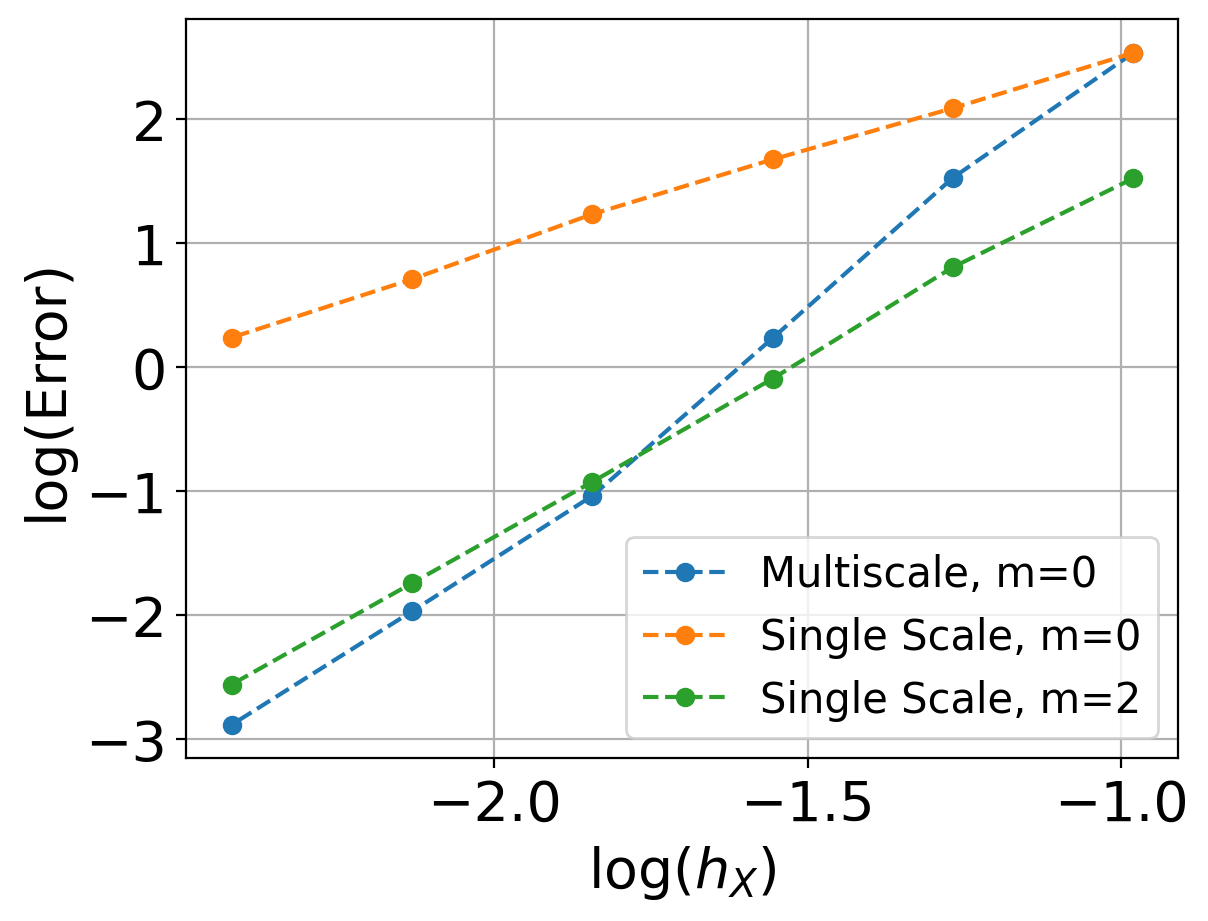}
			\caption{Error rates}
		\end{subfigure} \quad
		\begin{subfigure}{0.45\textwidth}
    \includegraphics[width=.97\textwidth]{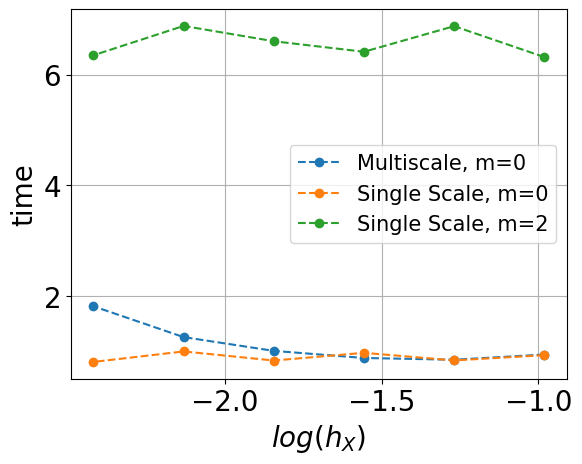}
			\caption{Timing in seconds} \label{subfig:timing}
		\end{subfigure}
    \caption{A comparison of error decay rates and running time between constant reproduction multiscale, and two quasi-interpolations based on constant reproduction and quadratic reproduction, for the approximation of the function $g$ of~\eqref{eq: approximated functions}. }
    \label{fig: const_multi_vs_quadratic_single}
\end{figure}

\begin{remark} \label{rem:QIcomplexity}
	\rev{The Construction of quasi-interpolation with a polynomial reproduction degree higher than $0$ usually involves solving linear systems, as in~\eqref{03-mlssystem}. Typically, these linear systems tend to be ill-conditioned. While the problem of ill-conditioning can be softened using properly shifted and scaled basis functions, see, e.g.~\cite{mirzaei2012generalized}, it cannot entirely be avoided as it also heavily depends on the geometry of the data sites. Moreover, the calculation is still significantly more expensive than the simple Shepard's method, see Figure~\ref{subfig:timing}. 
	
	For a more analytical aspect,} we briefly explain the complexity of quasi-interpolation for the above degrees of polynomial reproduction. First, to approximate $Q(f)$ at point $x$, we use data of points $x_i\in \operatorname{supp}\phi_\delta(\|x-\cdot\|_2)$. These points are the neighbours of $x$ as they satisfy $\|x-x_i\|_2\leq \delta$, where $\delta$ is the support radius of $\phi_\delta$. 
	
	We organize the sites $X$ in a k-d tree, with cost $O(dN \log N)$, where $N=\#X$, and $d$ is the dimension of the domain. For a point $x$, a query of points in a radius $\delta$, using the k-d tree, costs $O(\log N)$. Denote the complexity of weight calculation with $\omega$. In constant reproduction $\omega=1$. Else, in higher degree we solve a linear system meaning $\omega=Q^3(Q+\kappa)$, where $Q = \dim \left( \pi_m(\R^d)\right) = \binom{m+d}{d}$ (see, e.g., \cite[Theorem 2.5]{wendland_2004}. Due to quasi-uniformity, for any $\delta$ there is a bound $\kappa$ to the number of points in the $\delta$-neighbourhood of $x$. So, finally, the complexity of the quasi-interpolation of a single point $x$ is: $O(\log N + \omega \kappa)$.
\end{remark}

\subsection{Anomaly detection over scattered data by multiscale} \label{subsec:anomaly_detection}

\rev{When applying a local quasi-interpolation operator, the approximation error is related to the local smoothness of the approximated function. When based on local quasi-interpolation, the multiscale method uses several steps of local approximations. It thus provides a powerful alternative to analyze the function's local smoothness, each time with respect to the associated level. In other words, the multiscale approach \rev{appears also to offer} a tool for detecting non-smoothness. In turn, this analysis mechanism shows the efficiency of anomaly detection, similar to what we may find in classical literature, for example, in~\cite{mallat1992singularity}. We use this approach to detect anomalies based on the approximation error of our multiscale method.}

We demonstrate the above by approximating a function that we contaminate with a minor, synthetic anomaly. The approximated function is:
\begin{equation*}
    \label{eq: synthetic anomaly}
    \tilde{f}(x, y) =\begin{cases}
      1.01 \cdot f(x,y), & (x, y) \in [0.1, 0.25]\times[0.2, 0.4],\\
      f(x,y), & \operatorname{else},
\end{cases}
\end{equation*}
where $f$ is from~\eqref{eq: approximation func}. The approximation's configuration is the same as before, using Halton points as our scattered samples. We show in Figure~\ref{fig:synthetic anomaly} the fifth scale's error distribution. Indeed it is easy to see that the anomaly region is easily detected using the multiscale's error.

\begin{figure}
    \centering
    \begin{subfigure}[b]{.31\textwidth}
		\centering
		\includegraphics[width=\textwidth]{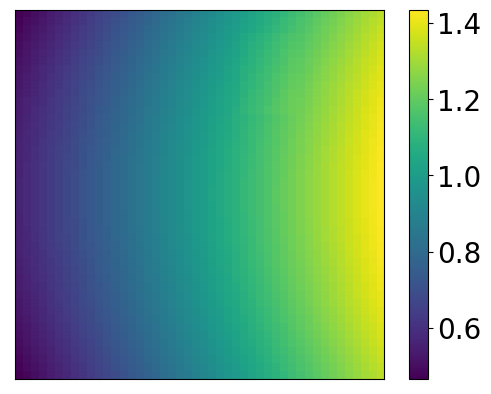}
		\caption*{The function $\tilde{f}$}
	\end{subfigure}  \hspace{4pt}
    \begin{subfigure}[b]{.32\textwidth}
		\centering
		\includegraphics[width=\textwidth]{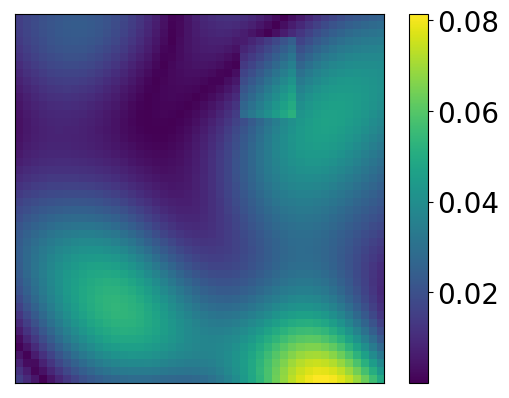}
		\caption*{First level error: $|\tilde{f}-\tilde{f}_1|$}
	\end{subfigure} 
	\begin{subfigure}[b]{.33\textwidth}
		\centering
		\includegraphics[width=\textwidth]{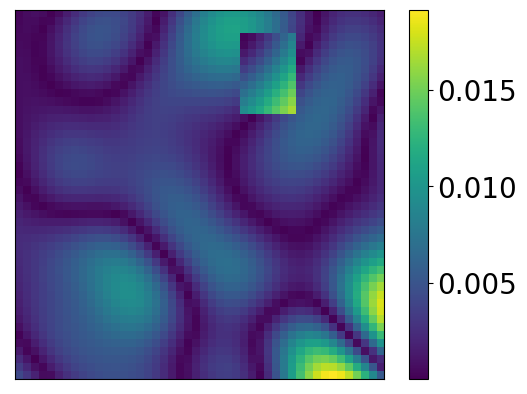}
		\caption*{Third level error: $|\tilde{f}-\tilde{f}_3|$}
	\end{subfigure} \\ \vspace{10pt}
	\begin{subfigure}[b]{.325\textwidth}
		\centering
		\includegraphics[width=\textwidth]{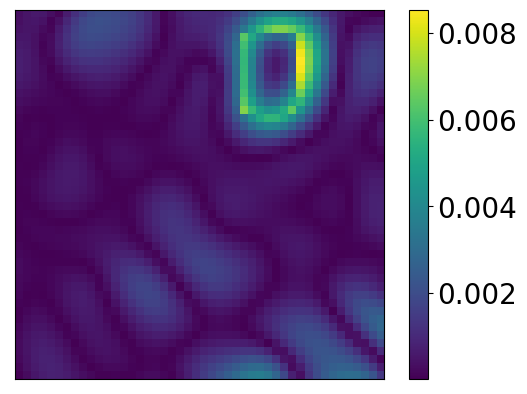}
		\caption*{Fifth level error: $|\tilde{f}-\tilde{f}_5|$ \\ (multiscale)} 
	\end{subfigure} \qquad 
	    \begin{subfigure}[b]{.33\textwidth}
		\centering
		\includegraphics[width=\textwidth]{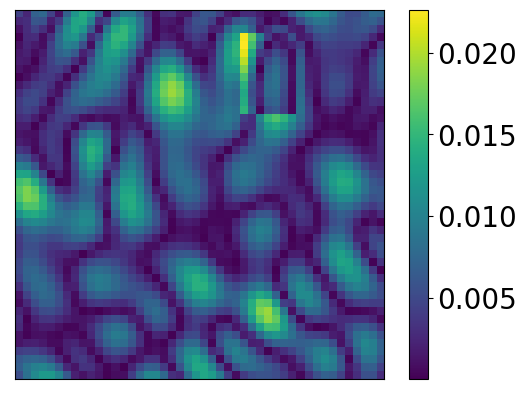}
		\caption*{Quasi-interpolation error: \\ $|\tilde{f}-Q_{X_5}(\tilde{f})|$ \quad (single scale)}
	\end{subfigure} 
    \caption{Anomaly detection using the multiscale approach. In the upper row, the function and two error maps correspond to the first and third levels of the multiscale. In the bottom row, the two error maps of the multiscale (left) and single scale of quasi-interpolation (right) correspond to the data of the fifth level. The anomaly is seen clearly thorough the multiscale error maps and especially over the last level. Detecting the same anomaly over the function (top left) or via the error map of the quasi-interpolation (bottom right) is much harder.}
    \label{fig:synthetic anomaly}
\end{figure}

\subsection{Scattered data approximation of manifold-valued functions}

We demonstrate our multiscale analysis based on quasi-interpolation for functions defined on $\R^2$ and with values on $\SO(3)$ (oriented rotations) and $\SPD(3)$ (symmetric positive definite matrices). We use the quasi-interpolation of the form~\eqref{eq: quasi-interpolation manifold}. The error between two functions $F$ anf $H$, analogous to~\eqref{eqn:error}, is 
\begin{equation*}
    \max_{x_i \in G(R, h)}\rho \left(F(x_i), H(x_i)\right), 
\end{equation*}
where the grid $G(R, h)$ has the same properties of the test grid in real valued scattered data.

\subsubsection{Approximation over oriented preserving rotations}
When evaluating~\eqref{eq: quasi-interpolation manifold}, we required to calculate Karcher means. For that, we use a iterative process suggested in~\cite{manton2004globally} and defined as,
\begin{equation} \label{eqn:iter_RCoM_SO}
    Y_{j+1} = \Exp{Y_j}{\frac{\sum_{i=1}^N {a_i\Log{Y_j}{F(x_i)}}}{\sum_{i=1}^N{a_i}}},
\end{equation}
where $Y_0$ is set to be one of the values $F(x_i)$, and we stop
iterating when $\rho(Y_{j+1},Y_j)<\varepsilon$ for a predefined small tolerance $\varepsilon>0$. \rev{Note that both \cite{karcher1977riemannian, manton2004globally} indicate linear convergence of the iterative process~\eqref{eqn:iter_RCoM_SO}. Therefore, the computational load of applying such a procedure depends on the accuracy needed. In our examples, we generally observed that calculating the Karcher mean by~\eqref{eqn:iter_RCoM_SO} was significantly more costly, sometimes by a factor of $10$ to $20$, depending on the prescribed  precision, compared to approximating the Karcher mean via finite procedures such as exp-log~\cite{grohs2012definability, rahman2005multiscale}, repeated binary averaging~\cite{dyn2017global, wallner2005convergence}, and inductive means~\cite{dyn2017manifold}. Therefore, these alternatives may be an excellent practical choice in some scenarios. However, in this section and the following tests, we keep the more accurate option, highlighting errors that stem from applying the approximation methods and not their manifold adaptation.}

Here, we provide an example where we approximate a function over rotations. The function is defined as follows by the Euler angles, which are taken via the xyz convention:
\begin{equation}
    \label{eq: so(3) function}
    F(x, y) = \operatorname{Euler}_{xyz}(1.2\sin{5x-0.1}, y^2/2-\sin{3x}, 1.5\cos{2x}).
\end{equation}
In practice, we form the $\SO(3)$ rotations from the Euler angles
using the python package scipy~\cite{2020SciPy-NMeth}. We visualize
the function by plotting the application of it, as field of
transformations, on a fixed vector, \rev{chosen to be the north pole of the unit sphere in $\mathbb{R}^3$}. This illustration appears at the
most left image of Figure~\ref{fig:rotations}. As in previous
examples, we compare our multiscale method with a direct application of the quasi-interpolation, \rev{that is a method similar to the suggested in~\cite{grohs2017scattered}}. The samples are scattered over our planer domain, as we demonstrate in the middle image of Figure~\ref{fig:rotations}. The comparison graph appears as the most right image of Figure~\ref{fig:rotations}, where we see how the multiscale error decays faster. 
\begin{figure}[ht]
	\centering
	\begin{subfigure}[b]{.3\textwidth}
		\centering
		\includegraphics[width = \textwidth]{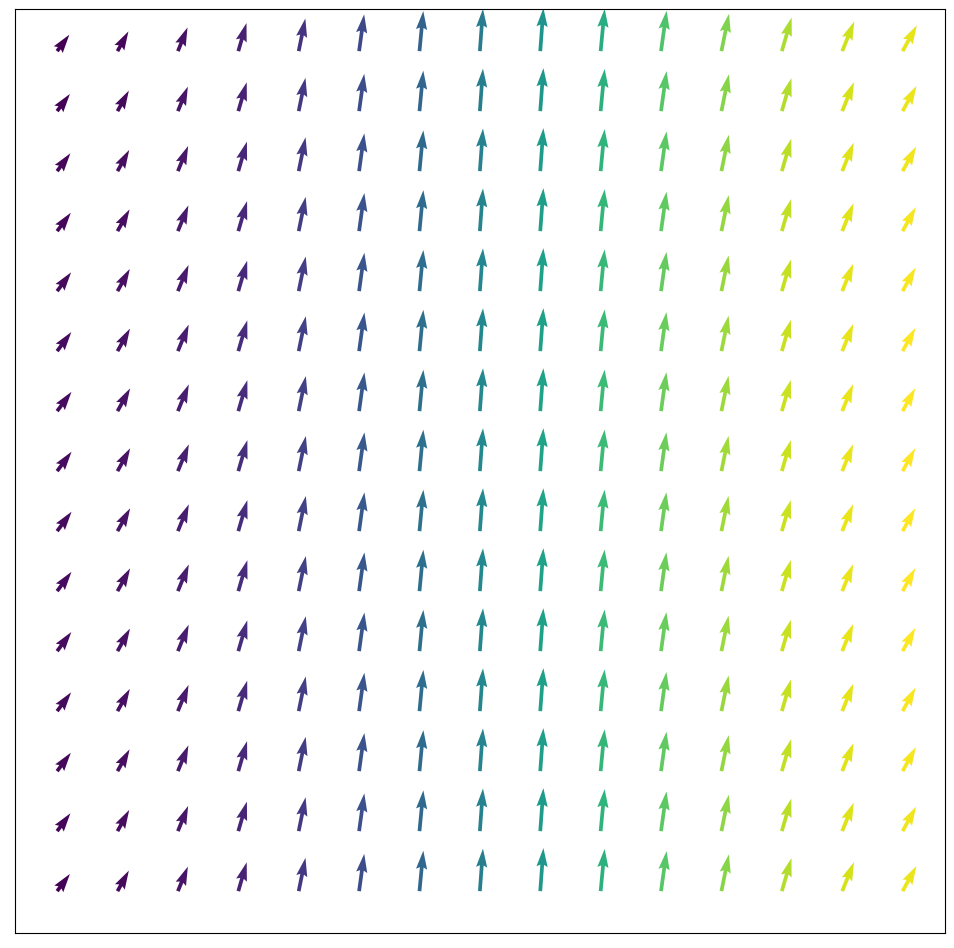}
		\caption*{The function $F$ of~\eqref{eq: so(3) function}}
		\label{fig: rotations original}
	\end{subfigure} \quad
        \begin{subfigure}[b]{.295\textwidth}
        \centering
        \includegraphics[width = \textwidth]{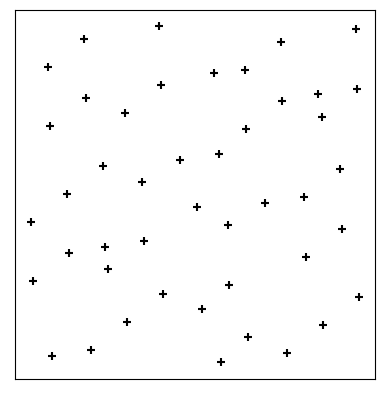}
        \caption*{Scattered sites of $X_1$}
        \end{subfigure}  \quad
		\begin{subfigure}[b]{.33\textwidth}
		\centering
		\includegraphics[width = \textwidth]{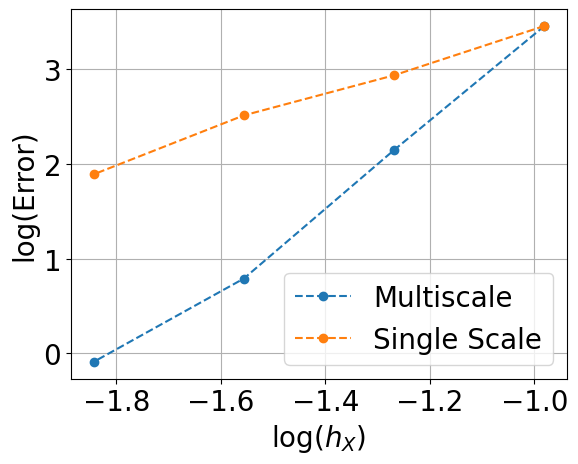}
		\caption*{Error rates}

	\end{subfigure} \hspace{2pt}
	\caption{Comparison of approximating a field of rotations over a planar domain. The rotations in $F$ are visualized by the effect of $F(x, y)$ on a fixed vector \rev{$\hat{x} = (1,0,0) \in \mathbb{R}^3$,} where the color represents the $z$ element $F(x, y)(\hat{x})_z$. The approximation sites are scattered at each level as a Halton sequence.
	On the left: the function to be approximated. In the middle: a demonstration of the scattered data sites of the first level. On the right: the approximation error rates.}
	\label{fig:rotations}
\end{figure}

\subsubsection{Approximation over symmetric positive definite matrices}

\begin{figure}[!hb]
	\centering
	\begin{subfigure}[t]{.3\textwidth}
		\centering
		\includegraphics[width = \textwidth]{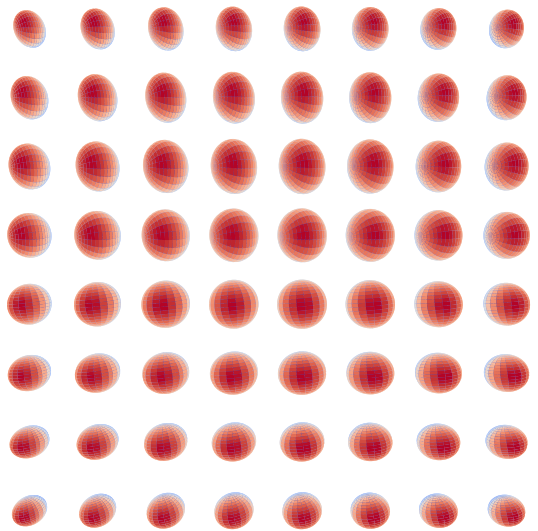}
		\caption*{The function $F$ of~\eqref{eq: spd(3) function}}
	\end{subfigure} \qquad
	\begin{subfigure}[t]{.38\textwidth}
		\centering
		\includegraphics[width = \textwidth]{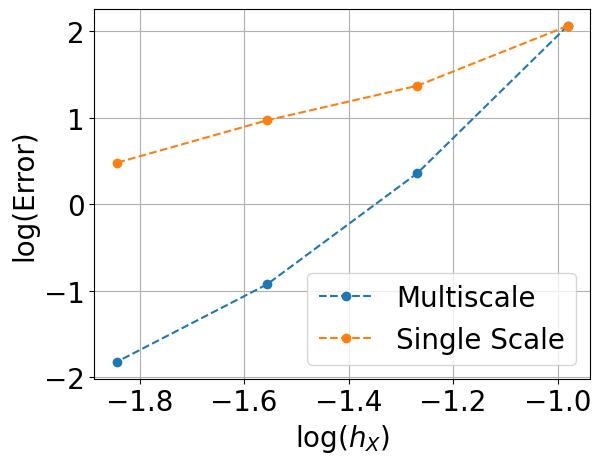}
		\caption*{Error rates}
	\end{subfigure} \hspace{2pt}
	\caption{Comparison of approximating a field of SPD matrices over a planar domain. The matrices are visualized by their associated centered ellipsoids. The approximation sites are scattered at each level as a Halton sequence. On the left: the function to be approximated. On the right: the approximation error rates.}
	\label{fig:spd}
\end{figure}

Recall that in $\SPD(d)$, the exponential and logarithms get the symmetric form
\begin{equation*}
    \Log{X}{Y} = X^{1/2}\log{\left(X^{-1/2}Y X^{-1/2}\right)}X^{1/2}, \quad \Exp{X}{y} =  X^{1/2}\exp{\left(X^{-1/2}y X^{-1/2}\right)}X^{1/2},
\end{equation*}
where $\log$ and $\exp$ are the standard matrix functions. Then, the geodesic distance is
\begin{equation*}
    \rho(X, Y) = \norm{\log{\left(X^{-1/2}Y X^{-1/2}\right)}}_2,
\end{equation*}

To evaluate the Karcher means in the quasi-interpolation, we use iterations from~\cite{Iannazzo2019} and define:
\begin{equation*}
    Y_{j+1}=\Exp{Y_j}{\theta_j\sum_{i=1}^N{a_i\Log{Y_j}{F(x_i)}}},\quad \theta_j = \frac{2}{\sum_{i=1}^N{a_i(x)\frac{c_i^{(j)}+1}{c_i^{(j)}-1}}},
\end{equation*}
where $c_i^{(j)}$ is the condition number of $Y_j^{-1/2}F(x_i)Y_j^{-1/2}$, and $Y_0 = \sum_{i=1}^N {a_i F(x_i)}$, which is a symmetric positive definite matrix as long as $a_i\geq0$, which is the case for our Shepard’s method~\eqref{eq: shepard}.

In the current example, we approximate a function $F$ of $\SPD(3)$, of the form
\begin{equation}
	\label{eq: spd(3) function}
	F(x, y) = G(x, y) + \left(G(x, y) \right)^T ,
\end{equation}
where $G(x, y) = \left(\abs{\cos{2y}+0.6}\right)e^{-x^2-y^2}\left(5I+A
\right) +I$, $A=\begin{pmatrix}
	\sin{5y} && y && xy \\
	0 && 0 && y^2 \\
	0 && 0 && 0
\end{pmatrix}$. We use the representation of any SPD matrix as a centered ellipsoid. This ellipsoid has main axes that are determined by the eigenvectors of the matrix and their lengths are the associated eigenvalues. Then, an illustration of $F$ of~\eqref{eq: spd(3) function} is given as the left image of Figure~\ref{fig:spd}.

Again, we compare the error decay of the multiscale approximation versus the quasi-interpolation. The data sites where taken again, scattered as a Halton sequence and the error decay rates are provided in the right image of Figure~\ref{fig:spd}. Also here, the multiscale's error decays faster.

\subsubsection{Application of denoising a field of rotations}
We summarize the numerical section with an application of our manifold-valued scattered data approximation. This application is a denoising process where we use the multiscale error rates as indicators for noise removal. In particular, and as done in classical signal processing, we adopt a procedure where we diminish the levels of approximation where error drops below a fixed threshold.

In our example, we use the same field of rotation $F$ of~\eqref{eq: so(3) function}. Then, we contaminate the function by changing each value with its product against the exponential of a matrix which was sampled from a normal distribution with a fixed variance $\sigma^2$. In other words, the random matrix is calculated in the Lie algebra and then projected back to the group. The samples are taken in a scattered-fashion from a Halton sequence

\rev{We follow a classical denoising process based on thresholding, see ,e.g.,~\cite{donoho1995noising}. Our} denoising process starts after the first approximation level; then, we filter out from the sites set $X_n$ all the sites with error larger than $t\cdot \max_{x\in{X_n}}{\norm{F(x)\ominus F_{n-1}(x)}}$, where $0\le t\le1$ is a fixed threshold.

The first results of the above denoising process appear in Figure~\ref{fig:ComparingSigmas}, where we present the error rates under various $\sigma$ values,\rev{ ranging from $0.01$ to $1$, which we recall being, as before, the variance of noise matrix in the tangent space.} As the variance of the noise increases, the denoising procedure struggles, and the associated decay rates implicitly indicate that more noise is presented in the samples or that the function is less ``smooth.'' From several tests, we also learn that the denoising process is not sensitive to choosing different threshold values and in particular that all tested values between $0.75 \le t \le 0.95$ lead to similar results.

\begin{figure}[ht]
    \centering
    \includegraphics[width = 0.4\textwidth]{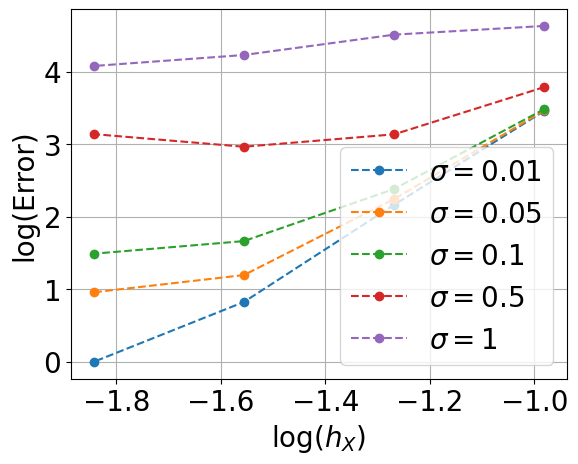}
    \caption{Approximation error rates, for our multiscale approximation, of fields of rotations contaminated with different levels of noise.	}
    	    \label{fig:ComparingSigmas}
\end{figure}

The visual result of the above denoising process appears in Figure~\ref{fig:Denoising Example}, where we present the original function (left), its noisy version (middle), and our denoised approximation (right).

\begin{figure}[ht]
    \centering
	\begin{subfigure}[b]{.32\textwidth}
    \centering
    \includegraphics[width=\linewidth]{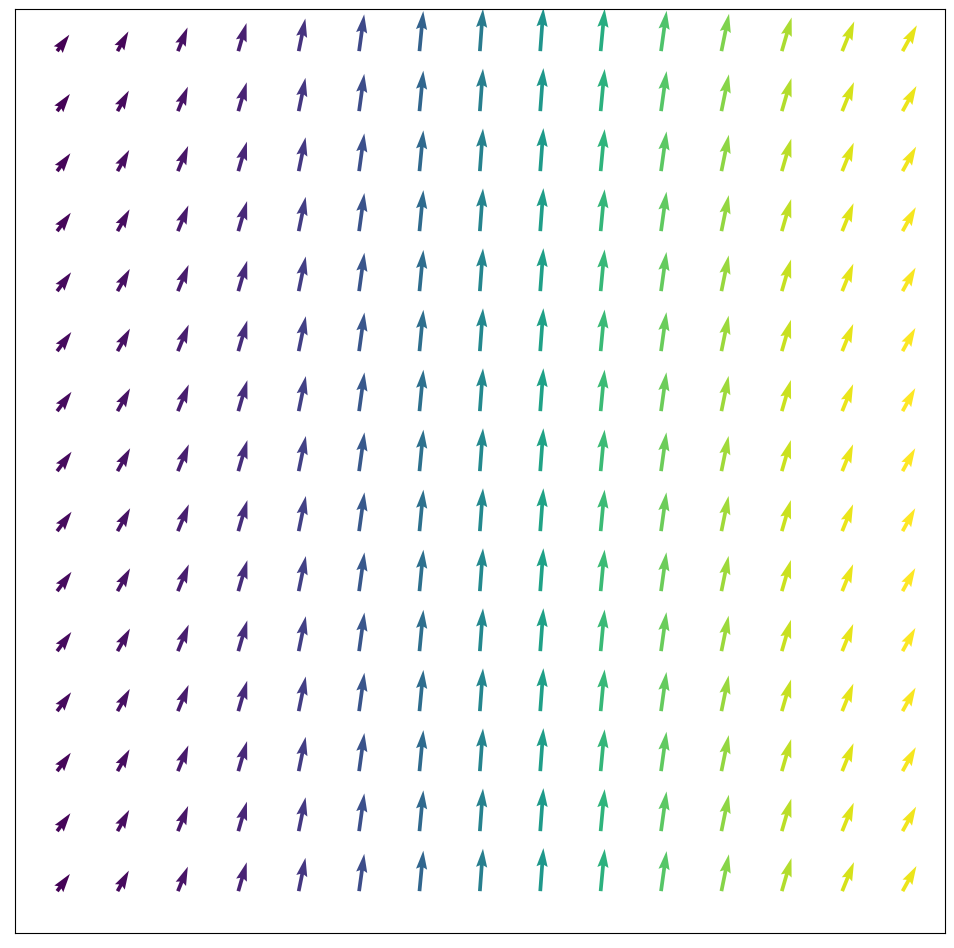}
    \caption*{$F(x, y)$}
    \label{fig:DenoisingOriginal}
    \end{subfigure}
\begin{subfigure}[b]{.32\textwidth}
    \centering
    \includegraphics[width=\linewidth]{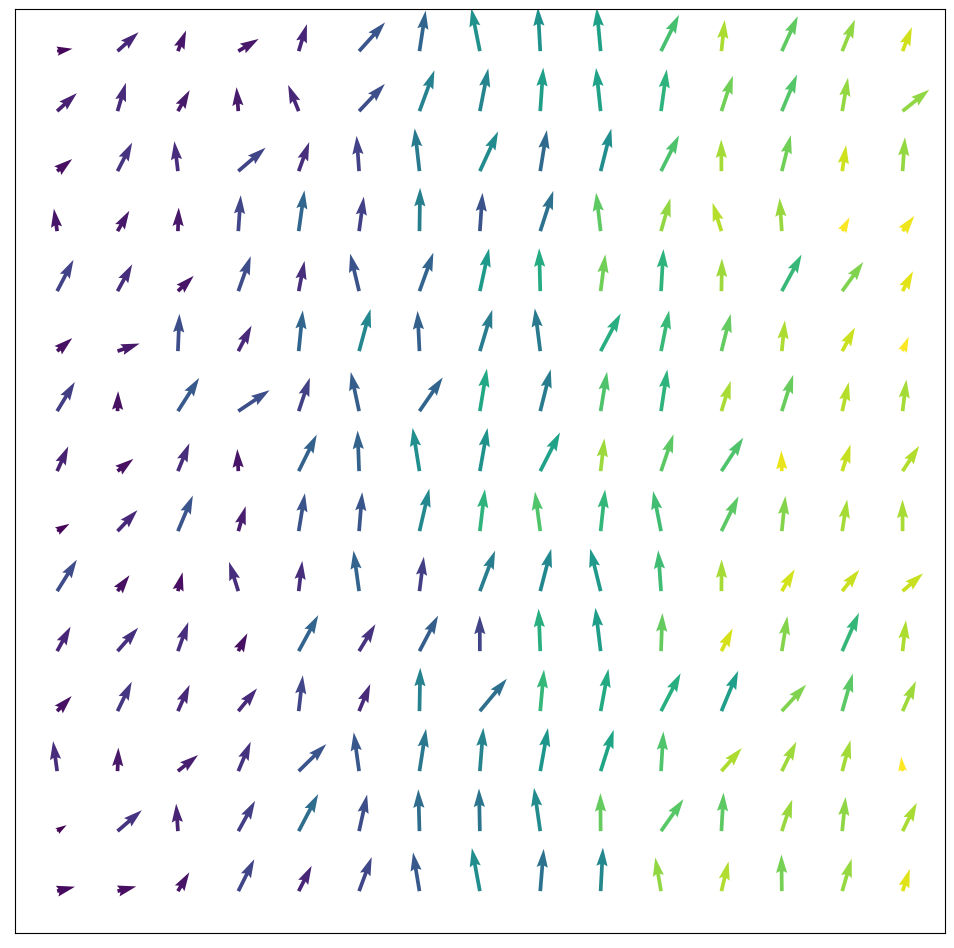}
    \caption*{The contaminated function}
    \label{fig:DenoisingNoise}
    \end{subfigure}
\begin{subfigure}[b]{.32\textwidth}
    \centering
    \includegraphics[width=\linewidth]{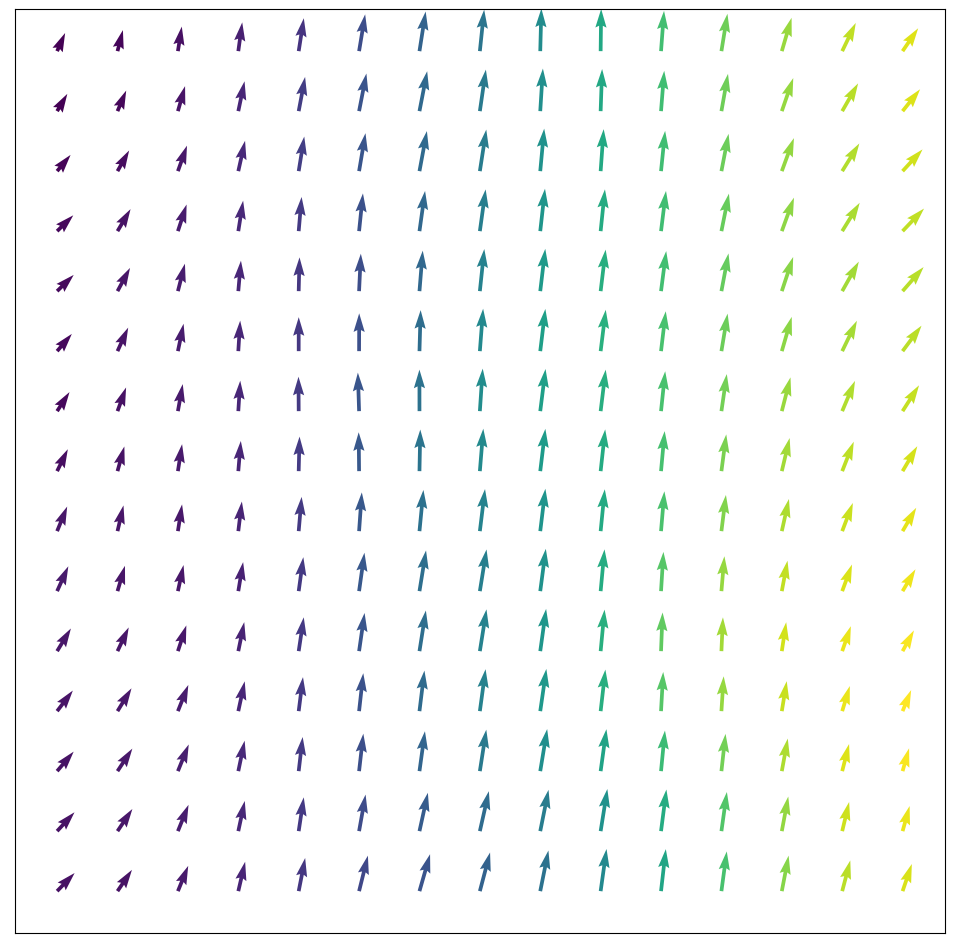}
    \caption*{Our approximation}
    \label{fig:DenoisingResult}
    \end{subfigure}
    \caption{Visualization of $F(x, y)$ of~\eqref{eq: so(3) function} (left), its noisy version with variance of $\sigma=0.2$ (middle) and the denoised approximation (right). The site picking is done with a threshold of $t=0.9$.}
    \label{fig:Denoising Example}
\end{figure}

\section*{Acknowledgments}
We thank Nira Dyn for her contribution to this project's conception and first steps. NS is partially supported by the NSF-BSF award 2019752. \rev{NS and HW are partially supported by the DFG award 514588180}.

\bibliographystyle{plain}
\bibliography{bibliography}


\appendix

\section{Implementation details of scattered data} \label{app: scattered data}

Our approximation is local and requires data from nearby sites. This is a result of using a local quasi-interpolation and it offers many advantages. When the sites (the function samples in the parameter space) are on a grid, it is easy to find the closest points just by indices calculation. However, when working with scattered points, finding the nearest points requires special attention.

To generate sites for our examples, we choose to use Halton sequence~\cite{halton1960efficiency}. Specifically, the Halton sequence is constructed using coprime numbers as its bases, where we use the bases $2, 3$. As our basic set, we generate such a Halton sequence with $n=400$ points and denote it by $H_{2,3}(400)$. Then, we scale and duplicate the sites, while keeping the ratio between the fill distance and the separation radius (see~\eqref{eqn:fill_dist_separation_rad}) fixed as $h_X / q_X = 4.01$.  Next, we scale it with $r = h_X / h_{H_{2,3}(400)}$, where $h_X$ is the requested fill distance. To fill a given domain $\Omega$, we cover it with translates to obtaining 
\begin{equation*}
    X = \bigcup_{i=0}^{\lceil \frac{x_{\max} - x_{\min}}{r} \rceil}\bigcup_{j=0}^{\lceil \frac{y_{\max} - y_{\min}}{r} \rceil} \Big\{ r H_{2,3}(400) + (x_{\operatorname{min}} + r i, y_{\operatorname{min}} + r j)  \Big\}   \quad .
\end{equation*} 
We show an example for such a set of scattered points in Figure~\ref{fig: halton}.
\begin{figure}[ht]
    \centering
    \includegraphics[width=0.5\linewidth]{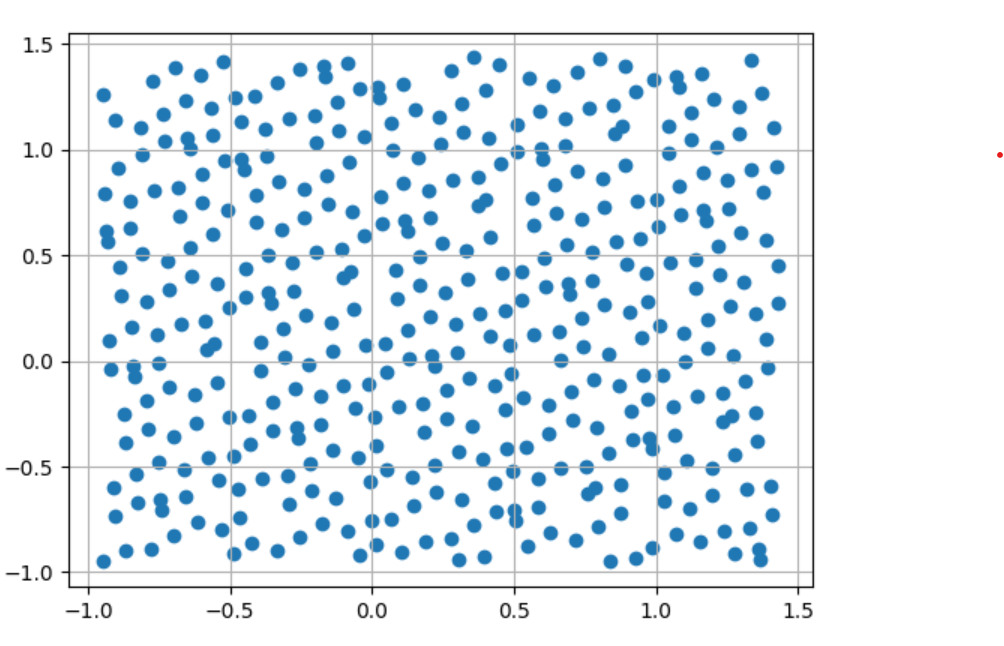}
    \caption{Halton points with $h_X=0.158$, which is equivalent to the fourth scale in our examples.}
    \label{fig: halton}
\end{figure}

Afterward, when the data sequence is generated, we use \textit{k-d trees} as the data structure for the scattered locations. In practice, our implementation utilizes the package \textit{{pykdtree}} for querying nearest neighbors within the RBF's support. Its theoretical performance is discussed in \cite[Section~14.2]{wendland_2004}.

\end{document}